\documentclass[a4paper, 11pt]{amsart}

\usepackage{latexsym}
\usepackage{amssymb}
\usepackage{amsmath}
\usepackage{amsthm}
\usepackage{amsfonts}
\usepackage{color}
\usepackage{enumitem}

\usepackage[all]{xy}

\usepackage{graphicx}
\usepackage{subfigure}

\usepackage{psfrag}
\usepackage{hyperref}
\usepackage{comment}

\newtheorem{thm}{Theorem}[section]
\newtheorem{cor}[thm]{Corollary}
\newtheorem{lem}[thm]{Lemma}
\newtheorem{prop}[thm]{Proposition}

\newtheorem{theorem}{Theorem}
\newtheorem{corollary}[theorem]{Corollary}

\theoremstyle{definition}
\newtheorem{defn}[thm]{Definition}

\newtheorem{rem}[thm]{Remark}

\numberwithin{equation}{section}

\newcommand{\scal}[1]{\langle #1 \rangle}

\newcommand{\RR}{\mathbb{R}}
\newcommand{\PP}{\mathbb{P}}
\newcommand{\QQ}{\mathbb{Q}}
\newcommand{\HH}{\mathbb{H}}
\newcommand{\CC}{\mathbb{C}}

\newcommand{\ZZ}{\mathbb{Z}}

\newcommand{\SO}{\mathrm{SO}}
\newcommand{\OO}{\mathrm{O}}
\newcommand{\U}{\mathrm{U}}

\newcommand{\Sp}{\mathsf{Sp}}
\renewcommand{\sp}{\mathfrak{sp}}

\newcommand{\GL}{\mathrm{GL}}
\newcommand{\Sym}{\mathrm{Sym}}
\newcommand{\Map}{\mathrm{Map}}

\newcommand{\diag}{\mathrm{diag}}

\newcommand{\Id}{\mathrm{Id}}
\newcommand{\sgn}{\operatorname{sgn}}
\newcommand{\inter}[1]{\mathrm{Int}\left(#1\right)}

\def\ol{\overline}

\newcommand{\lra}{\longrightarrow}
\newcommand{\lmt}{\longmapsto}
\newcommand{\ra}{\rightarrow}

\newcommand{\In}{\subseteq}

\DeclareMathOperator{\ind}{ind}

\newcommand{\sphere}{\mathrm{\mathbb{S}}}

\newcommand{\vf}{\mathfrak{X}}

\newcommand{\lp}{\Lambda}
\newcommand{\N}{\mathsf{N}}
\newcommand{\Cyl}{\mathcal{C}}
\newcommand{\E}{\mathcal{E}}
\renewcommand{\null}{\mathrm{null}}

\begin{document}



\title[On manifolds all of whose geodesics are closed]{On the Berger conjecture
for manifolds all of whose geodesics are closed }


\author{Marco Radeschi}

\author{Burkhard Wilking}
\date{\today}


\subjclass[2000]{}


\begin{abstract} We define a Besse manifold as a Riemannian manifold $(M,g)$ 
all of whose geodesics are closed. A conjecture of Berger states that
  all prime geodesics have the same length for any simply connected Besse manifold.
 We firstly show that the energy function on the free loop space of a simply connected Besse manifold 
 is a perfect Morse-Bott function  with respect to a suitable cohomology.
 Secondly we explain when the negative bundles along the critical manifolds are orientable. 
 These two general results, then lead to 
 a solution of Berger's conjecture when the underlying manifold 
 is a sphere of dimension at least four.
 \end{abstract}

\maketitle



\tableofcontents

Riemannian manifolds with all geodesics closed have been object of study since the beginning of the XX century,
when the first nontrivial examples were produced by Tannery and Zoll. 
To this day the famous book of Besse \cite{Be} describes 
the state of art to a large extent with a few but notable exceptions. 
We define a Besse metric $g$ on a manifold $M$, as a Riemannian metric all of whose geodesics 
are closed and a Besse manifold as a manifold endowed with a Besse metric.

The ``trivial'' examples of Besse manifolds  are the compact rank one symmetric spaces
(also called CROSSes), with their canonical metrics: round spheres $\sphere^n$, complex and quaternionic
projective spaces $\CC\PP^n$, $\HH\PP^n$ with their Fubini-Study  metric, and the Cayley projective
plane $Ca\PP^2$. To this day, these are also the only manifolds that are known to admit a Besse metric.
Moreover, on the projective spaces the canonical metric is the only known Besse metric.
On the other hand, in the case of spheres many other Besse metrics have been discovered by Zoll,
Berger, Funk, and Weinstein.

Given a Besse  manifold, a theorem of Wadsley states that all the prime geodesics
(i.e. which are not iterates of a shorter one) have a common period $L$. Using this result, a
combination of theorems by Bott-Samelson \cite[Thm 7.37]{Be} and McCleary \cite[Cor. A]{McC} proves
that any simply connected Besse manifold $M$ has the integral cohomology ring of a CROSS, 
the so called {\em model} of $M$.

A conjecture of Berger states that on a simply connected Besse manifold all
the prime geodesics have the \emph{same} length $L$. This strengthens Wadsley's result in the simply connected case.
If all prime geodesics have the same length one knows, for example, that the volume of $M$ is an integer multiple
of the volume of a round sphere of radius $L/2\pi$ \cite{We0}.

This conjecture was proved for 2 spheres by Gromoll and Grove \cite{GG}. Partial results in dimension three have been obtained by Olsen in \cite{Ols}. The main goal of this paper is to
prove the conjecture for all topological spheres of dimension at least 4.

\begin{theorem}\label{T:main}
Let $M$ be a topological sphere of dimension $\geq4$. Then, for every Besse metric on $M$,
all prime geodesics have the same length.
\end{theorem}

To prove Theorem \ref{T:main}, we study the free loop space $\lp M
=\mathcal{H}^1(\sphere^1,M)$
(the Sobolev space of maps $\sphere^1\to M$) of any Besse manifold.
In particular, we study the energy functional $E\colon \lp M\to \RR$, 
which associates to any free loop $\gamma\colon S^1\to M$ the energy
$$E(\gamma)=\int_{S^1}\|\gamma'(t)\|^2\textrm{dt}.$$ 
The basic idea is to  combine the simplicity of the topology of the 
the free loop space of $M$ with Morse theory in order to show that 
the critical levels of $E$ must be simple. To that end we prove two general results 
which hold on all orientable Besse manifolds.
The second named author proved in \cite{Wil} that $E$ is a Morse-Bott function.
In particular, the critical points of $E$, which precisely correspond to the closed geodesics in $M$,
form smooth ``critical'' submanifolds. Moreover, along every critical submanifold $C$ the negative eigenspaces
of the Hessian of $E$ give rise to a ``negative bundle'' $\N\to C$, of finite rank.

\begin{theorem}\label{T:orientability}
Let $M$ be an orientable Besse manifold, $C\In \lp M$  a critical submanifold 
of the energy functional, and $A\colon C\to \SO(n)$ denote the map sending a closed geodesic $\gamma$
to the holonomy map $A_\gamma\in \SO(T_{\gamma(0)}M)\cong \SO(n)$ along $\gamma$.
Then for any curve $c:\sphere^1\to C$ the negative
bundle over $c$ is orientable if and only if $A\circ c$ is nullhomotopic in $\SO(n)$.
\end{theorem}

The proof of Theorem~\ref{T:orientability} allows in a way 
to use the Poincare map to define a natural orientation on the negative bundle. 
One might hope to use similar ideas in other contexts,
e.g. to define (under suitable assumptions) a natural orientation on the unstable manifold
along a periodic Reeb orbit.
If $M$ is a spin manifold, the holonomy $A\colon C\to \SO(n)$ always lifts to a map
$\tilde{A}\colon C\to \textrm{Spin}(n)$, and in particular $A$ is always contractible. It follows in particular:

\begin{corollary}\label{C:orientability-spin}
Let $M$ be an orientable spin Besse manifold. Then for every critical submanifold $C\In \lp M$ 
for the energy functional, the negative bundle is orientable.
\end{corollary}

Almost all Besse manifolds are spin: in fact,
this is the case if the model of $M$ is not given by $\CC\PP^{2m}$. 
Since $M$ has the same integral cohomology as its model 
this follows either from  $H^2(M;\ZZ_2)=0$ or in case of the model $\CC\PP^{2m+1}$ from 
the fact that then $M$ is homotopy equivalent to its model and 
the homotopy invariance of Stiefel-Whitney classes 
$w_2(M)=w_2(\CC\PP^{2m+1})=0$.

Recall that $\OO(2)$ acts on $\lp M$ by reparametrization, and
the energy functional is $\OO(2)$-invariant. In particular,
we can use $E$ as an $\sphere^1$-equivariant Morse-Bott function, with $\sphere^1=\SO(2)\In\OO(2)$.
Let $M\to \lp M$ denote the embedding sending $p\in M$ to the constant curve $\gamma\equiv p$, and
denote the image of such embedding again by $M$.

\begin{theorem}\label{T:perfectness}
Let $M$ be a simply connected Besse manifold. Then the energy functional is a perfect Morse-Bott function
for the rational, $\sphere^1$-equivariant cohomology of the pair $(\lp M, M)$.
\end{theorem}

The proof of the perfectness of the energy functional uses three main ingredients.
\begin{enumerate}
\item \emph{Index parity}: On an orientable Besse manifold  $M$ the index of every closed geodesic has 
the same parity as $\dim M+1$.
\item \emph{Lacunarity}: If a manifold $M$ has the integral cohomology of a CROSS then, up to inverting
a finite number of primes, the cohomology $H^q_{\sphere^1}(\lp M,M;\ZZ)$ is zero whenever $q$ has
the same parity of $\dim M$.
\item \emph{Index gap}: If $c$ is a closed geodesic on a simply connected Besse manifold  $M$ 
of length $iL/q$, where $L$ is the common period of all geodesics and $i,q\in\mathbb{N}$,
 then, roughly speaking,
the differences $\ind c^q - \ind c^{q'}$, $q'\neq q$, are bounded away from zero by some constant independent
from $c$.
\end{enumerate}

The Index parity was proved by the second-named author in \cite{Wil}.
The Lacunarity and the Index gap are proved in the present paper.

The paper is structured as follows. Section \ref{S:Orientability} contains the proof of
Theorem \ref{T:orientability}. Several technical results are proved in the appendix.
Section \ref{S:free-loop-space-CROSS} and \ref{S:index-gap} contain the precise statements and
proofs of the Lacunarity and the Index gap, respectively. Finally, in Section \ref{S:perfectness}
Theorem \ref{T:perfectness} is proved, while Section \ref{S:proof-main-thm} contains the proof
of Theorem \ref{T:main}.


\section{Free loop space, and Morse-Bott theory}\label{S:recall}

Let $M$ be a Besse manifold. As defined above a Besse manifold is a Riemannian manifold 
all of whose geodesics are closed. 
Given $\sphere^1=\RR/\ZZ$, we define
the \emph{free loop space of $M$}, to be the Hilbert manifold $\lp M=\mathcal{H}^1(\sphere^1,M)$ of
maps $c:\sphere^1\to M$ of class $\mathcal{H}^1$. The \emph{energy functional} $E:\lp M\to \RR$ is defined as
\[
E(c)=\int_{\sphere^1}\|c'(t)\|^2dt.
\]

As pointed out in \cite{Wil}, the energy functional is a Morse-Bott function, and any critical set $K$ of energy $e$ consists of all closed geodesics of length $\ell=\sqrt{e}$. 
 From now on, we will use the following notation:
\begin{itemize}
\item We denote by $0=e_0<e_1<\ldots$ the sequence of critical values of $E$ (or \emph{critical energies}).
\item For every critical energy $e_i$, we denote by $K_i=E^{-1}(e_i)$ the corresponding critical manifold, and $\lp^i=E^{-1}([0,e_i+\epsilon))$ the corresponding sub-level set.
\item We define the \emph{index} of $K_i$, $\ind (K_i)$ to be the number of negative eigenvalues of $\textrm{Hess}(E)|_K$, counted with multiplicity.
\item For every critical manifold $K_i$, denote by $\N_i\to K_i$ the subbundle of $\nu(K_i)$ consisting of the negative eigenspaces of $H_i=\textrm{Hess}(E)|_{\nu(K)}$. We call $\N_i$ the \emph{negative bundle} of $K_i$.
\end{itemize}

For $i=0$, the critical set $K_0$ consists of point curves and will be from now on identified with $M$. For each $i>0$, let $\N_i^{\leq 1}=\{v\in \N_i\mid \|v\|\leq 1\}$, for some choice of norm on $\N_i$. By Morse-Bott theory, each sublevel set $\lp^k$ is homotopically equivalent to the complex obtained by successfully attaching the unit disk bundles $\N_i^{\leq 1}$, $i=1,\ldots k$, to $M=K_0$.

If one were to know the relative cohomology of the pairs $(\lp^i,\lp^{i-1})$, it would be possible to reconstruct the cohomology groups $H^*(\lp^i;\ZZ)$ iteratively, using the long exact sequence of the pair $(\lp^i,\lp^{i-1})$, and in the limit one would reconstruct the cohomology of $\lp M$. By excision, the relative homology group of the pair $(\lp^i , \lp^{i-1})$ is equal to the relative homology group of $(\N_i^{\leq 1},\partial \N_i^{\leq 1})$. If $\N_i\to K_i$ is orientable, then by the Thom isomorphism the cohomology of $(\N_i^{\leq 1},\partial \N_i^{\leq 1})$ equals the shifted cohomology of $K_i$:
\[
H^k(\lp^i,\lp^{i-1};\ZZ)\simeq H^k(\N_i^{\leq 1},\partial \N_i^{\leq 1};\ZZ)\simeq H^{k-\ind (K_i)}(K_i;\ZZ).
\]

Notice that the group $\sphere^1$ acts on $\lp M$ by reparametrization, and the energy functional is $\sphere^1$-invariant. In particular, the critical manifolds and the sublevel sets are all $\sphere^1$ invariant, and there is a natural $\sphere^1$-action on the negative bundles such that the maps $\N_i\to K_i$ are all $\sphere^1$-equivariant. Moreover, by choosing an $\sphere^1$-invariant metric on $\N_i$, the inclusion $(\N_i^{\leq 1},\partial \N_i^{\leq 1})\to (\lp^i,\lp^{i-1})$ is also $\sphere^1$-equivariant. In particular, using $\sphere^1$-equivariant cohomology we have that whenever $\N_i\to K_i$ is orientable, there is an isomorphism
\[
H^k_{\sphere^1}(\lp^i,\lp^{i-1};\ZZ)\simeq H_{\sphere^1}^{k-\ind (K_i)}(K_i;\ZZ).
\]

\subsection{Critical submanifolds}

 Given a Besse manifold $M$, by Wadsley Theorem \cite{Wad} the prime geodesics have a common period $L$, thus the geodesic flow induces a $\sphere^1$-action on the unit tangent bundle $T^1M$ of $M$. Moreover, if a critical set $K=E^{1}(e)$ contains prime geodesics, then the geodesics in $K$ have length $\ell=L/k$ for some integer $k$, and the map
 \[
 K_i\to T^1M,\qquad c\mapsto (c(0),c'(0)/\ell)
 \]
defines a diffeomorphism of $K_i$ into the subset of $T^1M$ fixed by $\ZZ_k\In \sphere^1$. This diffeomorphism is $\sphere^1$-equivariant, where $\sphere^1$ acts on $K_i$ by reparametrization, and on $T^1M$ by the geodesic flow. In particular if the length of the geodesics in $K$ is a multiple of $L$, then $K$ is diffeomorphic to $T^1M$.

In the rest of the paper we will use the following notation:
\begin{itemize}
\item We let $\{C_1,\ldots C_k\}$ denote the set of critical manifolds containing prime geodesics.
\item We call \emph{regular} any geodesic whose length is a multiple of $L$. Similarly, we say that a critical set is \emph{regular} if it contains regular geodesics.
\item Given a closed geodesic $c:[0,1]\to M$, we denote by $c^q:[0,1]\to M$, $c^q(t)=c(qt)$, the $q$-iterate of $c$.
\end{itemize}

Any critical set of positive energy can be written as $C^q=\{c^q\mid c\in C\}$ for some integer $q$ and some $C$ in  $\{C_1,\ldots, C_k\}$.

\section{Orientability of the negative bundle}
\label{S:Orientability}
Let $M$ be a Besse manifold. The goal of this section is to prove Theorem~\ref{T:orientability}.
The techniques used in this sections are completely independent from the remaining sections,
and therefore can be read independently or used as a black-box.

\subsection{Idea behind the proof} In \cite{Wil} the second-named author proved that the the Poincar\'e map
of a closed geodesic determines the parity of its index. In a way,
we will prove here the orientability of the negative bundle by showing that the Poincar\'e map of
a closed geodesic $c$ can help define a natural orientation on the negative space at $c$.

More precisely, consider a ``formal closed geodesic'', modelled by a map $R\in \Map([0,\pi],\Sym(n-1,\RR))$
(the curvature operator along the geodesic) and a map $A\in \OO(n-1)$ (the holonomy along the geodesic).
Given this data it is possible to define an ``index form'' $H$ 
for $(R,A)$, modelling the Hessian of the energy functional, and a ``negative space'' $\N$ defined as
the sum of the negative eigenspaces of $H$, which models the space of negative directions in the free loop space
(see Section \ref{SS:Alg-data}).

Choose a generic path $(R_{\tau}, A_{\tau})_{\tau\in [0,1]}$ of data from $(R,A)$ to the fixed data $(\Id, \Id)$. This path induces a family of index forms $H_\tau$, and a corresponding family of negative spaces $\N_\tau$. The idea is that, by fixing an orientation of $\N_1$, we want to induce an orientation on $\N_0=\N$ along the path. This can be easily done if $H_{\tau}$ does not change index along the path, in which case the collection $\{\N_\tau\}_{\tau\in [0,1]}$ defines a vector bundle on $[0,1]$. In general however $H_\tau$ does change index, and at the transition points $H_\tau$ has nontrivial kernel, as some eigenvalue of $H_\tau$ is changing sign.

For each $\tau$ it also makes sense to define a Poincar\'e map $P_{\tau}$ of the ``formal geodesic'' modelled by $(R_{\tau},A_{\tau})$, and one can consider the space
$$\mathcal{E}_\tau=\bigoplus_{\lambda\in (0,1)}E_{\lambda}(P_\tau),$$
with $E_{\lambda}(P_\tau)$ the eigenspace of $P_{\tau}$ of eigenvalue $\lambda$. Just like in the geometric setting, there is a natural identification between the kernel of $H_{\tau}$ (given by ``periodic Jacobi fields'') and the eigenspace $E_1(P_{\tau})$. Moreover, it was proved in \cite{Wil} that the dimension of $\mathcal{E}_\tau\oplus \N_\tau$ has constant parity at generic points. If $P_{\tau}$ is generic enough, it has eigenvalue 1 of geometric multiplicity at most 1, and therefore at every transition point both the index of $H_\tau$ (which equals the dimension of $\N_\tau$) and the number of real eigenvalues of $P$ in $(0,1)$ (which equals the dimension of $\mathcal{E}_\tau$) only change by 1.

In other words: at any transition point, the form $H_\tau$ will develop a one-dimensional kernel $V_H$, the Poincar\'e map $P_\tau$ will develop a one-dimensional fixed space $V_P$, such that there is a natural identification $V_H\simeq V_P$. Thus, at transition points, the sum $\N_\tau\oplus \mathcal{E}_\tau$ either gains or loses two one-dimensional subspaces that can be canonically identified. Since the sum of two identical subspaces carries a natural orientation, there is an obvious way to induce an orientation from $\N_{\tau-\epsilon}\oplus\mathcal{E}_{\tau-\epsilon}$ to $\N_{\tau+\epsilon}\oplus\mathcal{E}_{\tau+\epsilon}$. In our situation, we have $\mathcal{E}_0=0=\mathcal{E}_1$ and therefore the orientation on $\N_{1}\oplus \mathcal{E}_1=\N_1$ induces naturally an orientation on $\N_0\oplus\mathcal{E}_0=\N$.

Of course one needs to prove that the induced orientation on $N$ does not depend on the path chosen, and  therefore one needs a 2-dimensional variation, in which case more problem arise because the dimension of $\mathcal{E}_\tau$ can jump for other reasons (two real eigenvalues of $P_\tau$ could ``collide and disappear'', for example), which should be taken into account.

\subsection{The plan}

We sketch here the main steps of the proof: in Section \ref{SS:Alg-data} we define certain algebraic data
$(R_x,A_x)_{x\in \mathcal{X}}$ parametrized by a manifold $\mathcal{X}$, which formally model geometric structures
around
a closed geodesic. Moreover we show how one can construct, out of these data, a map $\N\to \mathcal{X}$
which resembles a vector bundle (we call these \emph{pseudo vector bundles}).


In Section \ref{SS:geom->alg} we consider a Besse manifold $M$ and a loop $c_s:\sphere^1\to \lp M$ in a critical manifold for the energy functional. We then explain how to associate an algebraic data set $(R_s,A_s)_{s\in \sphere^1}$ to $c_s$, and that proving the orientability of the negative bundle along $c_s$ is equivalent to proving the orientability of the pseudo vector bundle $\N\to \sphere^1$ induced by the algebraic data.

At first we assume that $A:\sphere^1\to \SO(n-1)$ is nullhomotopic, leaving the other case to the last section \ref{SS:not-nullhom}. In Sections \ref{SS:variation} and \ref{S:genericity} we produce a one-parameter deformation $$(R_{s,\tau},A_{s,\tau})_{(s,\tau)\in \sphere^1\times [0,1]}$$ of algebraic data, which is the original one for $\tau=0$ and has trivial negative bundle for $\tau=1$. This variation induces a pseudo vector bundle $\N\to \sphere^1\times [0,1]$, whose restriction to $\sphere^1\times \{0\}$ is the original one, and whose restriction to $\sphere^1\times\{1\}$ is orientable since the bundle is trivial.

In section \ref{S:modified-negative-bundle} we define a ``modified'' pseudo vector bundle $\N\oplus\E\to  \sphere^1\times [0,1]$, whose restriction to $\sphere^1\times\{0\}$ and $\sphere^1\times\{1\}$ coincides with $\N$.

In Section \ref{SS:local-orientability} we prove that a notion of orientability can be defined for $\N\oplus \E$, in such a way that $(\N\oplus \E)|_{\sphere^1\times\{0\}}$ is orientable if and only if $(\N\oplus \E)|_{\sphere^1\times\{1\}}$.

Finally, in Section \ref{SS:not-nullhom}, we discuss the case in which  $A:\sphere^1\to \SO(n-1)$ is not nullhomotopic. Using the results of the previous sections, we prove that in this case the negative bundle is not orientable.


\subsection{Algebraic data}\label{SS:Alg-data}

Let $M$ be a Riemannian manifold and $c:[0,2\pi]\to M$ a closed geodesic. Parallel translation along $c$ allows to identify the spaces $c'(t)^\perp$ with $V=c'(0)^\perp$, and it defines a map $A\in \OO(V)$ defined by $A(e(0))=e(2\pi)$ for every parallel normal vector field $e(t)$ along $c$. Moreover, the curvature operator of $M$ defines a map $R:[0,2\pi]\to \Sym^2(V)$ by
$$\scal{R(t)\cdot e_1(0),e_2(0)}=\scal{R(e_1(t),c'(t))c'(t),e_2(t)}$$
for every parallel normal vector fields $e_1(t),\,e_2(t)$ along $c$.

We say that the data $(R,A)$ model a ``formal geodesics'', because out of such data one can formally recover a number of objects usually related to real geodesics.

In fact, given the data of an Euclidean vector space $V\simeq \RR^{n-1}$, a piecewise continuous map $R\in \Map([0,2\pi],\Sym^2(n-1))$ and a $A\in \OO(n-1)$, we can define:

\begin{itemize}
\item The space $\vf=\{X:[0,2\pi]\to V\mid X(2\pi)=A\cdot X(0)\}$ of \emph{periodic vector fields}.
\item The space $\mathcal{J}=\{J:[0,2\pi]\to V\mid J''(t)+R(t)\cdot J(t)=0\}$ of \emph{Jacobi fields}.
\item The \emph{Poincar\'e map} $P:V\oplus V\to V\oplus V$ which sends $(u,v)$ to $(A^{-1}\cdot J(2\pi), A^{-1}\cdot J'(2\pi))$ where $J$ is the unique Jacobi field with $J(0)=u$, $J'(0)=v$. This map preserves the symplectic form $\omega((u_1,v_1),(u_2,v_2))=\scal{u_1,v_2}-\scal{u_2,v_1}$ and thus $P\in \Sp(n-1,\omega)$.
\item The \emph{index operator} $H:\vf\times \vf\to \RR$ defined as
\[
H(X,Y)=\int_0^{2\pi}\scal{X'(t),Y'(t)}-\scal{R(t)(X(t)),Y(t)}\, dt
\]
\end{itemize}
The index operator is symmetric, and by standard arguments it can be checked that the sum $\N\In \vf$ of the negative eigenspaces of $H$ is finite dimensional. We call this sum the \emph{negative space} of the pair $(R,A)$.

Given a manifold $\mathcal{X}$ and maps
\[
A:\mathcal{X}\to \SO(n-1),\quad R:\mathcal{X}\to \Map([0,2\pi],\Sym^2(n-1))
\]
then for every $x\in \mathcal{X}$ the data $(R_x, A_x)$ determine, in particular, a Poincar\'e map $P_x\in \Sp(n-1,\omega)$ an index form $H_x$ and a negative space $\N_x\in \vf$.

This determines a map $P:\mathcal{X}\to \Sp(n-1,\omega)$ and a space $\N=\coprod_{x\in \mathcal{X}}\N_x$ with a projection $\N\to\mathcal{X}$ sending $\N_x$ to $x$. The space $\N$ has a natural topology, induced by the inclusion $\N\In\mathcal{X}\times\Map([0,2\pi],V)$, and $\N\to \mathcal{X}$ has the structure of a fiberwise vector space.

In general, however, the map $\N\to \mathcal{X}$ is not a vector bundle, since the dimension of the fibers might change from point to point.

\begin{rem}\label{R:HtoP}
Notice that, given data $(R_x,A_x)_{x\in \mathcal{X}}$, there is always an isomorphism $\ker H_x\to E_1(P_x)$ ($E_1(P_x)$ the eigenspace of $P_x$ with eigenvalue 1). In fact, a vector field $X$ in $\ker H_x$ is a periodic Jacobi field, in the sense that $X''(t)+R_x(t)\cdot X(t)=0$ and $(X(2\pi)),X'(2\pi))=(A\cdot X(0),A\cdot X'(0))$. In particular, the vector $(X(0),X'(0))$ is a fixed vector for $P_x$ and therefore the map $X\mapsto (X(0),X'(0))$ defines the isomorphism $\ker H_x\to E_1(P_x)$.

In particular, when dimension of the fibers of $\N\to \mathcal{X}$ changes, at the transition point the kernel of $H_x$ must be nontrivial and therefore $P_x$ must have eigenvalue 1. This will be very important later on.
\end{rem}

\subsection{Paths of closed geodesics}\label{SS:geom->alg}
Let $C\In \lp M$ be a critical submanifold of the energy functional $E$ and let $c:\sphere^1\to C$, $s\mapsto c_s$ be a closed curve. Each $c_s$ defines
a unit speed closed geodesic which, possibly after scaling, can be parametrized as $c_s:[0,2\pi]\to M$.

Along $\alpha(s)=c_s(0)$, let $\{e_1(s),\ldots e_{n-1}(s)\}$ be an orthonormal frame of $c'_s(0)^{\perp}$ such that $e_i(0)=e_i(1)$, $i=1,\ldots n-1$. Moreover, let $e_i(s,t)$ denote the parallel transport of $e_i(s)$ along $c_s$. For each $s\in [0,1]$, let $A_s\in \SO(n-1)$ denote the holonomy map along $c_s$, i.e. $e_i(s,2\pi)=A_s\cdot e_i(s,0)$. Using the frame $\{e_1(s,t), \ldots e_{n-1}(s,t)\}$ we identify each space $c_s'(t)^{\perp}$ with $V=c_0'(0)^{\perp}\simeq\RR^{n-1}$.

We can see $A_s$ and $R_s(t)=R(\cdot, \dot{c}_s(t))\dot{c}_s(t)$ as maps
\begin{align*}
A&:\sphere^1\to \SO(n-1)\\
R&:\sphere^1\to \Map([0,2\pi],\,\Sym^2(n-1))
\end{align*}
and by the previous section, we can associate to the data $(R_s, A_s)_{s\in\sphere^1}$ a Poincar\'e map $P:\sphere^1\to \Sp(n-1,\omega)$ and a negative bundle $\N\to \sphere^1$. In this case the negative bundle is indeed a vector bundle, which coincides with the usual definition of negative bundle in Morse-Bott theory. In particular, the goal of this section is to prove that $\N\to \sphere^1$ is orientable.

\subsection{The variation}\label{SS:variation}

As previously explained, we want to produce a deformation $(R_{(s,\tau)}, A_{(s,\tau)})$, $(s,\tau)\in \sphere^1\times[0,1]$, such that $(R_{(s,0)}, A_{(s,0)})$ is related to the data of our geometric setup, $(R_{(s,1)}, A_{(s,1)})$ has trivial negative bundle, and in such a way that we can somehow ``keep track'' of the negative bundle along the deformation. In this section we emphasize the properties of the variation near $\tau=0$ and $\tau=1$, while in the next sections we concentrate on the interior of $\sphere^1\times [0,1]$.

Before defining the deformation, we slightly modify the initial data. We replace the $t$-interval $[0,2\pi]$ with a longer interval $[0,6\pi]$, and define $A_{(s,0)}=A_s$, and
\[
R_{(s,0)}(t)=\left\{\begin{array}{ll}R_s(t) & t\in [0,2\pi]\\ 4\cdot Id & t\in(2\pi, 4\pi]\\ Id & t\in (4\pi,6\pi]\end{array}\right.
\]
In general, given real numbers $0<L_1\ldots < L_k$ and curvature operators $R_i\in \Map((L_{i-1},L_i],\Sym^2(n-1))$, $i=1,\ldots k$ we define the concatenation $R_1\star\ldots\star  R_k\in \Map((0,L_k], \Sym^2(n-1))$ to be the operator whose restriction to $(L_{i-1},L_i]$ is $R_i$ for every $i=1,\ldots k$. In our case can then write
\[
R_{(s,0)}(t)=R_s(t)\star (4\; \Id)\star \Id.
\]
Notice that $R$ does not need to be continuous in $t$, but only piecewise continuous.

The Poincar\'e map of $(R_s,A_s)$ is the same as the one of $(R_{(s,0)}, A_{(s,0)})$, while the negative bundle of  $(R_{(s,0)}, A_{(s,0)})$ is a sum of the negative bundle of  $(R_s,A_s)$ with the (trivial) negative bundle of $(4\,\Id\star\Id,\Id)$. In particular, it will be enough to prove that the negative bundle of $(R_{(s,0)}, A_{(s,0)})$ is orientable.
\\

Notice that $A_{(s,0)}$ defines a loop in $\SO(n-1)$, which is contractible if the manifold is spin (for example, when $M=\sphere^n$ or $\HH\PP^n$), and let
\[
H_{(s,\tau)}:\sphere^1\times [0,1]\to \SO(n-1)
\]
be a homotopy with $H_{(s,0)}=A_{(s,0)}$ and $H_{(1,s)}=\Id$.

Given a function $\varphi:[0,1]\to \RR$ such that
\[
\left\{\begin{array}{ll}
\varphi(\tau)\equiv 0     & \textrm{around } \tau=0  \\
\varphi(\tau)\equiv 1     & \textrm{around } \tau=1  \\
\varphi'(\tau)\geq 0    &   \forall \tau\in (0,1)
\end{array}
\right.
\]

we now define a variation $(\hat{R}_{(s,\tau)}, A_{(s,\tau)})$ by
\begin{align}
\hat{R}_{(s,\tau)} &= \big((1-\varphi(\tau))R_s+\varphi(\tau)\Id\big)\star\big((4-3\tau)\cdot \Id\big) \label{E:R-variation}\\
A_{(s,\tau)} &= H_{(\varphi(\tau),s)}
\end{align}


Let us set the notation and call
\[
\Cyl=\sphere^1\times [0,1]
\]
the space parametrised by $s$ and $\tau$. The variation $(\hat{R}_{(s,\tau)}, A_{(s,\tau)})$ can be seen as parametrised data $(\hat{R}_x,A_x)_{x\in \Cyl}$, as defined in Section \ref{SS:Alg-data}. In particular, we have an associated Poincar\'e map $\hat{P}:\Cyl\to \Sp(n-1,\omega)$. For reasons that will be clearer later, we want to have some control over $\hat{P}$, and in particular over its real eigenvalues. At the moment we do not have such control, although in the following lemma we show that, sufficiently close to the boundary of $\Cyl$, the map $\hat{P}$ has no real eigenvalues.

\begin{lem}\label{L:around-the-boundary}
There is a neighbourhood $U$ of $\partial\Cyl$ such that the Poincar\'e map $\hat{P}$ of $(\hat{R}_x,A_x)_{x\in \Cyl}$ does not have any real eigenvalue on $U\setminus \partial \Cyl$.
\end{lem}
\begin{proof}
It is enough to show that for every $s\in \sphere^1$, $P_{(s,\tau)}$ does not have eigenvalues in $(0,1)$ in some neighbourhoods of $\tau=0$ and $\tau=1$.
For $\tau=0$, $\hat{P}_{(s,0)}=P_s$ is the Poincar\'e map of a closed geodesic in our geometric situation, and we know that $P_{s}^k=Id$ for some $k$. In particular the eigenvalues of $P_{s}$ are roots of unity. Since the Poincar\'e map depends continuously on $\tau$, its eigenvalues vary continuously as well and thus we only have to check how the eigenvalue 1 of $P_{s}$ disappears for any $\tau>0$ small.

Recall that we constructed the variation $(\hat{R}_{(s,\tau)},A_{(s,\tau)})$ so that for small values of $\tau$, $A_{(s,\tau)}=A_{s}$ and $\hat{R}_{(s,\tau)}$ only changes in the $t$-interval $[2\pi,4\pi]$, thus it is not hard to prove that, with respect to the canonical basis of $V\oplus V=\RR^{2(n-1)}$, $\hat{P}_{(s,\tau)}$ can be written as
\[
\hat{P}_{(s,\tau)}=\Delta A_{s}^{-1}\cdot (\cos \theta(\tau)\, \Id+\sin \theta(\tau)\, J)\cdot \Delta A_{s}\cdot P_s
\]
where $J=\left(\begin{array}{cc} & -Id_V \\ Id_V& \end{array}\right)$, $\Delta A_s=\diag(A_s,A_s)$ and $\theta(\tau)=2\pi\sqrt{4-3\tau}$.

Since we are only interested in the eigenvalues of $\hat{P}_{(s,\tau)}$ we can, up to conjugation of $\hat{P}_{(s,\tau)}$ and $P_s$ with $\Delta A_s$, reduce ourselves to the case
\[
\hat{P}_{(s,\tau)}=(\cos \theta(\tau) \Id+\sin\theta(\tau) J)\cdot P_s
\]
We claim that $\hat{P}_{(s,\tau)}$ cannot have real eigenvalues for small $\tau$.
Suppose in fact that $\hat{P}_{(s,\tau)}$ has an eigenvector $v_{\tau}$ with real eigenvalue $\lambda_{\tau}$, and $v_\tau$ tends to a fixed point of $P_s$ for $\tau\to 0^+$. Around $\tau=0$, the map $\tau\mapsto v_\tau$ is smooth. Differentiating the equation $\hat{P}_{(s,\tau)}v_{\tau}=\lambda_\tau v_\tau$ and taking the limit as $\tau\to 0$ we obtain
\[
\theta'(0)\cdot Jv_0+ P_sv'_0=\lambda'_0v_0+ v_0'.
\]
By evaluating the two sides of the equation using the symplectic form $\omega( \cdot, v_0)$ we get
$$\theta'(0)\cdot\omega(Jv_0, v_0)+\omega(P_sv'_0,v_0)=\omega(v'_0,v_0).$$
Since $v_0$ is a fixed point for $P_s$, we get 
$\omega(P_sv'_0,v_0)=\omega(v'_0,v_0)$ and the equation becomes $\theta'(0)\cdot\omega(Jv_0, v_0)=0$, which is not possible since $\theta'(0)\neq 0$ and $\omega(Jv_0,v_0)\neq 0$.

The case around $\tau=1$ can be handled in a similar fashion.
\end{proof}

The goal of the next section is to prove that we can produce a curvature operator
$$\tilde{R}:\Cyl\to \Map((4\pi,6\pi], \Sym^2(n-1)),$$
arbitrarily close to the constant map $\equiv \Id$ (and in fact equal to the constant map in a neighbourhood of $\partial \Cyl$), such that the Poincar\'e map of $(\hat{R}_x\star \tilde{R}_x,A_x)_{x\in \Cyl}$ is controlled.


\subsection{Making the Poincar\'e map generic}\label{S:genericity}

By \cite[Lemma 3.8]{Wil} there are curvature operators
\[
R_0,R_1,\ldots R_h:[0,2\pi]\to \Sym^2(n-1),
\]
with $h=\dim \Sp(n-1,\omega)$, such that the map
\begin{align*}
(-1,1)^h&\lra \Sp(n-1,\RR)\\
(a_1,\ldots a_h)&\lmt \textrm{Poincar\'e map of } (R_0+a_1R_1+\ldots a_hR_h,\Id)
\end{align*}
is a diffeomorphism between $(-1,1)^h$ and a neighbourhood $U$ of the identity in $\Sp(n-1,\omega)$.
In particular, given a map $\tilde{P}:\Cyl\to U$ there is a map
\[
\tilde{R}:\Cyl\to \Map([0,2\pi],\Sym^2(n-1))
\]
such that $\tilde{P}$ is the Poincar\'e map of $(\tilde{R},\Id)$. 

It is easy to see in general that that if $P_1, P_2, P_3$ are the Poincar\'e maps of $(R_1,A_1)$, $(R_2,A_2)$ and $(R_1\star R_2, A_3)$ respectively, then
\[
(\Delta A_1\cdot P_1)\cdot (\Delta A_2\cdot P_2)=\Delta A_3\cdot P_3
\]
where $\Delta A=\diag(A,A)\in \Sp(n-1,\omega)$.

In our case, given $\hat{P}_{x}$ the Poincar\'e map of the pair $(\hat{R}_{x}, A_{x})_{x\in \Cyl}$ defined in Equation \eqref{E:R-variation}, then the Poincar\'e map of $(R_{x}\star \tilde{R}_{x},A_{x})$ is
\[
\Delta A^{-1}_{x}\cdot \tilde{P}_{x}\cdot \Delta A_x \cdot \hat{P}_x
\]
In particular, for any map $P:\Cyl\to \Sp(n-1,\omega)$ sufficiently close to $\hat{P}$, such that $P|_{\partial \Cyl}=\hat{P}|_{\partial \Cyl}$, we can find an operator $\tilde{R}:\Cyl\to \Map([0,2\pi],\Sym^2(n-1))$ sufficiently close to the identity, such that $\tilde{R}|_{\partial\Cyl}\equiv \Id$ and $(\hat{R}_x\star\tilde{R}_x,A_x)_{x\in \Cyl}$ has Poincar\'e map $P$.
\\

The goal of this section is to prove that for a ``generic'' choice of $\tilde{R}$, the Poincar\'e map $P:\Cyl\to \Sp(n-1,\omega)$ of $(\hat{R}_x\star\tilde{R}_x,A_x)_{x\in \Cyl}$ has the following properties:
\begin{itemize}
\item For every $x\in \Cyl$, the positive real eigenvalues of $P_x$ have geometric multiplicity 1.
\item The set of points $x\in \Cyl$, whose Poincar\'e map $P_x$ has eigenvalue 1, is a smooth subvariety.
\item The function $\bar{\chi}:\Cyl\times \RR_+\to \RR$, $\bar\chi(x,\lambda)=\det(P_x-\lambda\,\Id)$ does not have critical value $0$ in the interior of $\Cyl$.
\end{itemize}

To this end, consider the following subsets of $\Sp(n-1,\omega)\times \RR_+$:
\begin{align*}
\mathcal{G}&= \{(Q,\lambda)\in \Sp(n-1,\omega)\times \RR_+\mid  \dim \ker(Q-\lambda\, \Id)\leq 1\}\\
\mathcal{G}_1&= \{(Q,\lambda)\in \Sp(n-1,\omega)\times \RR_+\mid  \dim \ker(Q-\lambda\, \Id)= 1\}\\
\mathcal{G}_0&= \{(Q,1)\in \mathcal{G}_1\}=\mathcal{G}_1\cap \left(\Sp(n-1,\omega)\times\{1\}\right)
\end{align*}

Clearly $\mathcal{G}$ is open in $\Sp(n-1,\omega)\times \RR_+$.  We show in Proposition \ref{P:map-surj} and Lemma \ref{L:G0inG1} that $\mathcal{G}_1$ is a smooth hypersurface of $\mathcal{G}$ and $\mathcal{G}_0$ is a smooth hypersurface of $\mathcal{G}_1$.

\begin{lem}\label{L:generic-P}
In any neighbourhood of $\hat{P}:\Cyl\to \Sp(n-1,\omega)$ there exists a $P:\Cyl\to \Sp(n-1,\omega)$ such that the image of
\[
(P\times \Id)|_{\mathrm{int}(\Cyl)\times \RR_+}: \mathrm{int}(\Cyl)\times \RR_+\to \Sp(n-1,\omega)\times \RR_+
\]
($\mathrm{int}(\Cyl)$ denotes the interior of $\Cyl$) is contained in $\mathcal{G}$, and it intersects $\mathcal{G}_0$ and $\mathcal{G}_1$ transversely.
\end{lem}
\begin{proof}
Let $\Sp_1(n-1,\omega)$ denote  the set of symplectic matrices whose positive real eigenvalues have geometric multiplicity 1, and let $\Sp_0(n-1,\omega)\In \Sp_1(n-1,\omega)$ denote the subset of matrices with eigenvalue 1. By Proposition \ref{P:geom-mult} $\Sp_1(n-1,\omega)$ is open, and its complement has codimension $\geq 3$. Moreover, $\Sp_0(n-1,\omega)$ is a smooth algebraic subvariety, and it has codimension at least 1 in $\Sp_1(n-1,\omega)$.

Since $\dim\Cyl=2$ it is possible to find a $P:\Cyl\to \Sp_1(n-1,\omega)$, close to $\hat{P}$ and with $P|_{\partial \Cyl}=\hat{P}|_{\partial\Cyl}$, such that $P|_{\mathrm{int}(\Cyl)}$ intersects $\Sp_0(n-1,\omega)$ transversely. Moreover, we can do it in such a way that $P|_{\textrm{int}(\Cyl)}$ is an embedding (because $n\geq 3$) and $P(\textrm{int}(\Cyl))$ is disjoint from $P(\partial\Cyl)$. By construction, the image of $(P\times \Id)|_{\mathrm{int}(\Cyl)\times \RR_+}: \mathrm{int}(\Cyl)\times \RR_+\to \Sp(n-1,\omega)\times \RR_+$ lies in $\mathcal{G}$ and it intersects $\mathcal{G}_0$ transversely.

Let $\chi:\mathcal{G}\to \RR$ denote the function $\chi(Q,\lambda)=\det(Q-\lambda\, \Id)$. This function is a submersion by  Proposition \ref{P:map-surj}. Clearly $\mathcal{G}_1=\chi^{-1}(0)$ and $P\times \Id$ fails to meet $\mathcal{G}_0$ transversely if and only if $(P\times \Id)^*\chi: \Cyl\times \RR_+\to \RR$ has some critical point of value $0$. Suppose then that $(P\times \Id)^*\chi$ has critical value $0$. By Proposition \ref{P:vector-field}, there exists a vector field $V$ in $\Sp_1(n-1,\omega)$ such that $d_{\textsc{\tiny (Q,$\lambda$)}}\chi(V)>0$ for every $(Q,\lambda)$ in $\mathcal{G}$. We fix a nonnegative function $f$ supported on a neighbourhood of the image of $P$, which is $0$ on $P(\partial\Cyl)$ and let $\Phi_t$ be the flow of $f\,V$ for some time $t$. If $t$ is small enough, then $(P\circ \Phi_t)\times \Id$ does not have critical point $0$ and we obtain the result, up to replacing $P$ with $P\circ \Phi_t$.
\end{proof}

We will from now on consider a variation
\[
(R_x=\hat{R}_x\star\tilde{R}_x,A_x)_{x\in \Cyl}
\]
where $\hat{R}, A$ are defined in Equation \eqref{E:R-variation}, and $\tilde{R}$ is defined in such a way that the Poincar\'e map $P$ of $(R_x,A_x)_{x\in \Cyl}$ satisfies the conditions of Lemma \ref{L:generic-P}.


\subsection{The modified negative bundle}\label{S:modified-negative-bundle}

Recall from Section \ref{SS:Alg-data} that to the data $(R_x,A_{x})_{x\in \Cyl}$ there is an associated negative bundle $\N\to \Cyl$, which has the structure of a fiberwise vector space, but it is not a vector bundle in general since the dimension of the fibers might (and does, in general) change from point to point. Nevertheless, whenever the restriction of $\N$ to a subset of $U\In \Cyl$ has constant rank then $\N|_U\to U$ is a vector bundle in the usual sense. We make this precise with the following definition.

\begin{defn}
A \emph{pseudo vector bundle} consists of topological spaces $E, B$, a continuous surjection $\pi:E\to B$, a section $0:B\to E$ (the \emph{zero section}), a map $a:E\times_BE\to E$ (\emph{addition}) over $B$ and, for every $\lambda\in \RR$, an operation $\lambda\cdot :E\to E$ (\emph{scalar multiplication}) over $B$, satisfying the usual axioms of vector bundles. Moreover, we require that there exists an open dense set $B_{reg}$ of $B$ (the \emph{regular part}) and an open cover $\{U_i\}_{i\in I}$ of $B_{reg}$, such that $E|_{U_i}\simeq U_i\times \RR^{n_i}$ (where $n_i$ might depend on $U_i$).
\end{defn}

By abuse of language, we will sometimes denote a pseudo vector bundle by $\pi:E\to B$, or simply by $E$ when $B$ and $\pi$ are understood. It is clear that $\N$ admits the maps in the definition of pseudo vector bundle.

The following lemma shows that there exists an open dense set $\Cyl_{reg}\In \Cyl$ over which $\N$ is a vector bundle, thus proving that $\N\to \Cyl$ is a pseudo vector bundle.

\begin{lem} \label{L:equivalence}
Let $x\in \Cyl$. The following hold:
\begin{enumerate}
\item The index form $H_x$ has nontrivial kernel if and only if the Poincar\'e map $P_x$ has a fixed vector. Moreover, there is an isomorphism between the eigenspace of $P$ with eigenvalue 1,  $E_1(P_x)$, and $\ker(H_x)$. 
\item If  $\ker H_x=0$ for some $x\in \Cyl$, then $\N\to \Cyl$ is a vector bundle around $x$. It thus makes sense to define $\Cyl_{reg}=\{x\in \Cyl\mid \ker H_x=0\}$.
\item $\Cyl_{reg}$ is open and dense in $\Cyl$.
\end{enumerate}
\end{lem}
\begin{proof}
1) If $X\in \ker H_x$ then by, the first variation formula for the energy function, $X$ is a periodic Jacobi field and
\[
P_x(X(0),X'(0))=(X(6\pi),X'(6\pi))=(X(0),X'(0)).
\]
In particular, there is an isomorphism $\ol{\phi}_x: \ker H_x\to E_1(P_x)$, $\ol{\phi}(X)=(X(0),X'(0))$.

2) Since $H_x$ is a symmetric map it has real eigenvalues and, since they change continuously with $x$, it follows that the number of negative eigenvalues of $H_x$ remains constant if $H_x$ has empty kernel. Therefore, if $\ker H_x=0$ then $\N$ is a vector bundle around $x$.

3) By the previous points, the complement of $\Cyl_{reg}$ consists of the points $x\in \Cyl$ such that $P_x$ has eigenvalue 1, 
which by construction is closed and has codimension at least 1 in $\Cyl$.
\end{proof}

\begin{defn}
A pseudo vector bundle $\pi:E\to B$ \emph{locally orientable} if there exists a $\ZZ_2$ cover $O\to B$ such that, for every component $B_i$ of $B_{reg}$, the restriction $O|_{B_i}$ is isomorphic to the orientation bundle of $E|_{B_i}\to B_i$.
\end{defn}

By the construction of the variation $(R_x, A_x)$ there is some $\epsilon>0$ small enough, such that the restriction of $\N$ to the curves
\[
\gamma_0(s)=(\epsilon, s)\qquad \gamma_1(s)=(1-\epsilon,s)
\]
has constant rank and therefore $\N|_{\gamma_i}$, $i=0,1$, is a vector bundle. If we could prove that $\N$ is locally orientable, it would follow that $\N|_{\gamma_0}$ is orientable if and only if $\N|_{\gamma_1}$ is orientable. Since $\N|_{\gamma_0}$ isomorphic to the sum of the geometric negative bundle with a trivial bundle, while $\N|_{\gamma_1}$ is trivial, this would prove the orientability of the geometric negative bundle. 
\\

With this goal in mind, we introduce a second pseudo-vector bundle $\E\to \Cyl$, $\E\In \Cyl\times(V\oplus V)$, whose fiber at $x\in \Cyl$ is the vector space consisting of the eigenspaces of $P_x$, with eigenvalues in $(0,1)$. It makes sense to define a fiberwise direct sum
\[
\N\oplus \E\to \Cyl
\]
which we call \emph{modified negative bundle} and this is the bundle that we will consider from now on.

By Lemma \ref{L:around-the-boundary}, $\E=0$ around $\partial \Cyl$ and thus we are not changing anything along $\gamma_1$ and $\gamma_2$. In particular, we can prove the local orientability of $\N\oplus \E$ instead of $\N$, and we will still obtain that the geometric negative bundle is orientable. The next two sections are devoted to proving the local orientability of $\N\oplus \E$.


\subsection{Local orientability of pseudo vector bundles}

In this section we show some classes of pseudo vector bundles that are always locally orientable.

We start by remarking that local orientability of a pseudo vector bundle $\pi:E\to B$ can be proved by showing that there is an open cover $\{U_i\}$ of $B$ such that $E|_{U_i}$ is locally orientable, and such that the corresponding orientation bundles $O(U_i)$ agree on intersections.

Let $M, \, N$ be manifolds of the same dimension, let $E\to N$ be a vector bundle,
and let $f: N\to M$ a closed map with finite fibers. For each $p\in M$, let $F_p=\bigoplus_{q\in f^{-1}(p)}E_q$.
Define $f_*E$ as the space $f_*E=\coprod_{p\in M}F_p$ with projection $f_*\pi:F\to M$ sending $F_p$ to $p$. By Sard's theorem, the set $M_{reg}\In M$ of regular points for $f$ is open and dense. Around each point $x\in M_{reg}$ there is a neighbourhood $U$ with $f^{-1}(U)=\coprod U_i$, such that $f|_{U_i}:U_i\to U$ is a diffeomorphism, and there is a canonical bijection $\phi_U:f_*E|_{U}\to \bigoplus_iE|_{U_i}$. The space $f_*E$, endowed with the roughest topology that makes the maps $f_*\pi$ and $\phi_U$ continuous, is clearly a pseudo vector bundle, which we call the \emph{push-forward} of $E$ via $f$.

%
%

\begin{lem}\label{L:loc orient sum}
Given a pseudo vector bundle $E\to M$ and a vector bundle $F\to M$, $E\oplus F\to M$ is locally orientable if and only if $E\to M$ is locally orientable. 
\end{lem}
\begin{proof}
If $O_E\to M$, $O_F\to M$ are the orientation bundles of $E,\, F$ respectively, then the orientation bundle of $E\oplus F$ exists and it is given by $O_E\otimes O_F:=(O_E\times O_F)/\ZZ_2$, where $\ZZ_2$ acts diagonally. Therefore, $E\oplus F$ is locally orientable. On the other hand, if $O_{E\oplus F}\to M$, $O_F\to M$ are the orientation bundles of $E\oplus F$ and $E$ respectively, then $O_{E\oplus F}\otimes O_F$ is the orientation bundle of $E$.
\end{proof}

\begin{prop}\label{P:locally-orientable}
Let $M, \,N$ be manifolds without boundary of the same dimension, and let $f:N\to M$ be a closed map with finite fibers. Then for every vector bundle $E\to N$ the push forward $f_*E\to M$ is locally orientable.
\end{prop}
\begin{proof}
We will define a pre-sheaf on $\mathcal{O}$ on $N$ such that $\mathcal{O}(U)=\ZZ_2$ for any $U$ small enough, and whose restriction to any open set in  $N_{reg}$ is isomorphic to the sheaf of local sections of the orientation bundle. Then the \'etale space $O=\coprod_{p\in N}O_p$, $O_p=\varinjlim_{U\ni p}O(U)$, together with the projection $O\to N$ sending $O_p$ to $p$, will be the $\ZZ_2$-cover we need to prove the result.

For any $p\in M$, $f^{-1}(p)$ is a discrete set $\{q_1,\ldots q_r\}$. Given a small neighbourhood $U$ of $p$, $f^{-1}(U)$ is a disjoint union of neighbourhoods $U_i$ of $q_i$, $i=1,\ldots r$. Defining $f_i=f|_{U_i}:U_i\to U$, the preimage of $f_i$ has almost everywhere constant parity, and we define $\epsilon_i=0$ if given parity is even, and $\epsilon_i=1$ if it is odd.

Let $b=(b_1,\ldots b_r)$ be an $r$-tuple where each $b_i$ is a local basis of sections of $E|_{U_i}$, and let $\mathcal{B}$ denote the set of such $r$-tuples. Given two $r$-tuples $b, b'$, there exists an $r$-tuple $(J_1,\ldots J_r)$ where each $J_i$ is a local section of $\GL(E|_{U_i})$ taking $b_i$ to $b_i'$. We can finally define $\mathcal{O}(U)$ to be the set of maps $\theta: \mathcal{B}\to \{\pm 1\}$ such that for any $b,b'$ in $\mathcal{B}$,
\[
\theta(b)=\theta(b')\cdot\prod_{i=1}^r\sgn\left(\det J_i\right)^{\epsilon_i}.
\]
It is easy to see that if $U$ is a small neighbourhood of a point in $M_{reg}$ then $f^{-1}(U)$ is a disjoint union of open sets $U_1,\ldots U_r$ such that $f_i:U_i\to U$ is a homeomorphism. In particular $\epsilon_i=1$ for all $i=1,\ldots r$, and $\mathcal{O}(U)$ coincides with the set of sections of the orientation bundle of $f_*E|_U$.
\end{proof}


\subsection{Local orientability of $\N\oplus \mathcal{E}$}\label{SS:local-orientability}

In this section we prove that the modified vector bundle is locally orientable, by showing that it locally takes the form $E_1\oplus f_*E_2$ for some vector bundle $E_1\to \Cyl$ and some push forward $f_*E_2\to \Cyl$ with respect to some function $f:\ol{\Cyl}\to \Cyl$. In order to do this, we must first construct the space $\ol{\Cyl}$, which will be defined by gluing two spaces $\Cyl_P$ and $\Cyl_H$ along their (diffeomorphic) boundaries.
\\

Consider the subset $\Cyl_P\In \Cyl\times (0,1]$ defined as
\[
\Cyl_P=\{(x,\lambda)\in \Cyl\times (0,1]\mid \lambda\textrm{ is an eigenvalue of }P_x\}
\]
and let $f_P:\Cyl_P\to \Cyl$ denote the obvious projection.

\begin{prop} The following hold:
\begin{enumerate}
\item The set $\Cyl_P$ is a submanifold with boundary $\partial\, \Cyl_P=\Cyl_P\cap \Cyl\times\{1\}$.
\item There is a line bundle $E_P\to \Cyl_P$, such that $\mathcal{E}={f_P}_*E_P|_{\inter{\Cyl_P}}$, where $\inter{\Cyl_P}$ denotes the interior of $\Cyl_P$.
\end{enumerate}
\end{prop}

\begin{proof}
1) As in Section \ref{S:genericity}, let $\mathcal{G}_1\In \Sp(n-1,\RR)\times \RR_+$ denote the subset of couples $(Q,\lambda)$ where $\dim \ker (Q-\lambda \Id)= 1$, and let $\mathcal{G}_0$ denote the subset of couples $(Q,1)$ in $\mathcal{G}_1$. By the construction of $(R_x,A_x)_{x\in\Cyl}$ (cf. Lemma \ref{L:generic-P}) the image of
\[
P\times \Id: \Cyl\times \RR_+\lra \Sp(n-1,\RR)\times\RR_+
\]
intersects $\mathcal{G}_1$ and $\mathcal{G}_0$ transversely. In particular $(P\times \Id)^{-1}(\mathcal{G}_1)$ is a smooth hypersurface, $(P\times \Id)^{-1}(\mathcal{G}_0)$ is a smooth hypersurface in $(P\times \Id)^{-1}(\mathcal{G}_1)$ dividing it in two components, and it is easy to see that $\Cyl_P$ is one such component. In particular, $\Cyl_P$ is a smooth manifold with boundary $\partial \Cyl_P=(P\times \Id)^{-1}(\mathcal{G}_0)$.

2) By construction, for every $x\in \Cyl$, every positive real eigenvalue of $P_x$ has geometric 
multiplicity $1$ (cf. Lemma \ref{L:generic-P}). In particular, the space $E\In \Cyl_P\times (V\oplus V)$ given by
\[
E_P=\{((x,\lambda),v)\in \Cyl_P\times (V\oplus V)\mid P_xv=\lambda v\}
\]
with the obvious projection $E_P\to \Cyl_P$ is a line bundle, and by definition $\mathcal{E}={f_P}_*(E_P)|_{\inter{\Cyl_P}}$.
\end{proof}

By Lemma \ref{L:equivalence}, the set
\[
\Sigma=f_P(\partial\,\Cyl_P)
\]
coincides with the set of points $x$ such that $\ker H_x\neq 0$ and, therefore, the set of points
around which $\N$ is not a vector bundle.
%
%
Since $\ker H_x$ is isomorphic to the eigenspace of $P_x$ of eigenvalue 1, and this is 1 dimensional,
it follows that $\dim \ker H_x=1$ for every $x\in\Sigma$. Since $\Sigma$ is compact, it is possible to
find a neighbourhood $U_{\Sigma}$ of $\Sigma$ and some $\epsilon>0$ such that $H_x$ has at most one
eigenvalue in $(-\epsilon, \epsilon)$ for every $x$ in $U_\Sigma$. Consider now
\begin{align*}
\psi:& U_\Sigma \times (-\epsilon, \epsilon)\to \RR\qquad \psi(x,\lambda)=\det(H_x-\lambda I)
\end{align*}
and define $\Cyl_H=\big(U_\Sigma\times (-\epsilon,0]\big)\cap \psi^{-1}(0)$. Equivalently, $\Cyl_H$ is
the set of pairs $(x,\lambda)$ such that $\lambda$ is a nonpositive eigenvalue of $H_x$.
Let $f_H:\Cyl_H\to U_\Sigma$ denote the obvious projection, and let $E_H\to \Cyl_H$ be the line bundle
whose fiber at $(x,\lambda)$ is the (one dimensional by definition) eigenspace of $H_x$ with eigenvalue $\lambda$.

\begin{prop}\label{P:gluable}
The following hold:
\begin{enumerate}
\item There is a diffeomorphism $\phi:\partial\,\Cyl_P\to \Cyl_H\cap\big( \Cyl\times\{0\}\big)$ such that $f_H\circ \phi=f_P$.
\item $\Cyl_H$ is a smooth submanifold, with boundary $\partial\,\Cyl_H=\Cyl_H\cap \Cyl\times\{0\}$.
\item There is an isomorphism of line bundles
\[
\ol{\phi}:E_P|_{\partial\, \Cyl_P}\lra E_H|_{\partial\, \Cyl_H}
\]
over $\phi$.
\end{enumerate}
\end{prop}
\begin{proof}
1) The set $\Cyl_H\cap \big(\Cyl\times\{0\}\big)$ consists of points $(x,0)$ where $x\in U$ and  $\ker H_x\neq 0$. By Lemma \ref{L:equivalence} the map $\phi:\partial\,\Cyl_P\to\Cyl_H\cap \big(\Cyl\times\{0\}\big)$ sending $(x,1)$ to $(x,0)$ is a diffeomorphism. 

2) We first prove that $0$ is a regular value of  $\psi:U_\Sigma\times (-\epsilon, \epsilon)\to \RR$. Given $(x,\lambda)\in \psi^{-1}(0)$, $\lambda$ is an eigenvalue of $H_x$. By construction $\lambda$ has geometric multiplicity $1$ and, since $H_{x}$ is symmetric, it has algebraic multiplicity 1 as well. Therefore, by letting ${\partial\over\partial t}$ denote the vector at $(x,\lambda)$ tangent to $(-\epsilon,\epsilon)$, we have
\[
d_{(x,\lambda)}\psi\left({\partial\over\partial t}\right)= {d\over dt}\Big|_{t=\lambda}\det(H_x-tI)\neq 0.
\]
Then $\psi^{-1}(0)$ is a smooth submanifold of $U\times (-\epsilon, \epsilon)$.
By the previous point, the subset $(\Cyl\times\{0\})\cap \psi^{-1}(0)=\Cyl_H\cap (\Cyl\times\{0\})$ is 
a smooth submanifold of $\psi^{-1}(0)$ which divides it into two components, one of which is $\Cyl_H$.

3)  For any point $(x,1)\in \partial\, \Cyl_P$, we have a map $\ol{\phi}_{(x,1)}:E_P|_{(x,1)}\to E_H|_{(x,0)}$ 
sending the fixed point $(v,w)\in E_1(P_x)=E_P|_{(x,1)}$ of $P_x$ to the unique Jacobi
field $J\in \ker H_x=E_H|_{(x,0)}$ with $J(0)=v$, $J'(0)=w$.
\end{proof}

By Proposition \ref{P:gluable}, it makes sense to define:
\begin{itemize}
\item The manifold $\ol{\Cyl}=\Cyl_P\sqcup_\phi \Cyl_H$, without boundary.
\item The map $f=f_P\sqcup_\phi f_E:\ol{\Cyl}\to U$. It is easy to check that this map is closed.
\item The line bundle $E=E_P\sqcup_{\ol{\phi}}E_H\to \ol{\Cyl}$.
\end{itemize}
 
\begin{prop}
The modified negative bundle $\N\oplus \mathcal{E}$ is locally orientable.
\end{prop}
\begin{proof}
It is enough to argue locally around a point $p\in \Cyl$. If $p$ does not belong to $\Sigma$, then we can find a neighbourhood $U$ disjoint from $\Sigma$. Therefore, $\N|_U$ is a vector bundle, and $f_P:f_P^{-1}(U)\to U$ is a proper map with finite fibers, with $\mathcal{E}|_U=f_*E_P$. By Proposition \ref{P:locally-orientable} and Lemma  \ref{L:loc orient sum}, $(\N\oplus \mathcal{E})|_U$ is locally orientable.

If $p$ lies in $\Sigma$, we can take a neighbourhood $U$ of $p$ that is contained in $U_{\Sigma}$. Let $F\to U$ denote the subbundle of $\N$ whose fiber at $x\in U$ is the sum of the negative eigenspaces of $H_x$ with eigenvalue $\leq -\epsilon$. By the construction of $U$, $F$ is a vector bundle and
\[
\left(\N\oplus\mathcal{E}\right)|_U=F\oplus f_*E|_{f^{-1}(U)}.
\]
Again by Proposition \ref{P:locally-orientable} and Lemma  \ref{L:loc orient sum}, $\N\oplus \mathcal{E}|_U$ is locally orientable.
\end{proof}

\subsection{If $A_s$ is not nullhomotopic.}\label{SS:not-nullhom}
We just finished proving that given data $(R_s,A_s)_{s\in \sphere^1}$ with $A:\sphere^1\to \SO(n-1)$
nullhomotopic, the corresponding negative bundle is orientable. 
We now complete the proof of Theorem~\ref{T:orientability} by showing that $\N$ is not orientable
if $A:\sphere^1\to \SO(n-1)$ not nullhomotopic.
In fact, in this case we can replace the data $V, R_{s}, A_{s}$ with
\begin{align*}
V^+& = V\oplus \RR^2=\RR^{n+1}\\
R^+_{s} & = \diag(R_{s},Q_s)\\
A^+_{s} & = \diag(A_{s},\textrm{Rot}_{-s})
\end{align*}
where $\textrm{Rot}_s\in \SO(2)$ denotes rotation by $s$ and $Q_s\equiv Q:[0,2\pi] \to \Sym^2(2)$ is an operator (independent of $s\in \sphere^1$) with Poincar\'e map equal to
$$B=\left(\begin{array}{cc|cc} &  & -1 &  \\ &  &  & 1 \\ \hline1 &  &  &  \\ & -1 &  & \end{array}\right)$$

One can check that the Poincar\'e map $P_s$ of $(Q_s,\textrm{Rot}_{-s})_{s\in \sphere^1}$ is given by the product $\diag(\textrm{Rot}_s,\textrm{Rot}_s)\cdot B$. This family of Poincar\'e maps has constant eigenvalues $\pm i$, and in particular no eigenvalue 1. By Remark \ref{R:HtoP}, the index of  $(Q,\textrm{Rot}_{-s})_{s\in \sphere^1}$ remains constant.  Therefore the negative bundle $\N'\to \sphere^1$ of the data $(Q,\textrm{Rot}_{-s})_{s\in \sphere^1}$ is indeed a vector bundle and so is the negative bundle $\N^+=\N\oplus \N'$ of the data $(R^+_s,A^+_s)_{s\in \sphere^1}$.

Since $A^+_{s}$  now defines a contractible loop in $\SO(n+1)$, by the result above $\N^+$ is orientable. However, we claim that $\N'$ is not orientable, from which it follows that $\N$ is not orientable either. In fact, for every $s$ let $U_s$ denote the space of vector fields
$$
U_s=\{J\mid \exists t_0\in (0,2\pi)\textrm{ s.t. } J(0)=J(t_0)=0,\,J|_{[0,t_0]}\textrm{ $Q_s$-Jacobi,}\,J|_{[t_0,2\pi]}\equiv 0\}
$$
In this case, since $Q_s$ does not depend on $s\in\sphere^1$, neither does $U_s$, which then describes a trivial bundle over $\sphere^1$. Moreover, let $W_s$ denote a maximal space of $Q_s$-Jacobi fields $J:[0,2\pi]\to \RR^2$ such that
$$
J(0)=\textrm{Rot}_s(J(2\pi)),\, \scal{\textrm{Rot}_s(J'(2\pi))-J'(0),J(0)}<0.
$$
Using Equations (1.4') of \cite{BTZ} it is easy to show that $\N'_s$ is isomorphic to $U\oplus W_s$. In this case, the space $W_s$ can be chosen to be the (1-dimensional) space of spanned by the Jacobi field with initial conditions
\begin{align*}
J(0)&=\big(\sin(-{s/ 2}),\cos(-{s/ 2})\big)\\
J'(0)&=\big(\cos(-{s/2}), \sin(-{s/2})\big).
\end{align*}
In particular $W_s$ forms a non-orientable vector bundle and, since $U_s$ is a trivial vector bundle, $\N'\simeq U\oplus W$ is not orientable.


\section{Free loop spaces of cohomology CROSSes}
\label{S:free-loop-space-CROSS}

The goal of this section is to prove that for every manifold $M$ with the integral cohomology ring of a CROSS,
\[
H^i_{\sphere^1}(\lp M, M;R)=0 \qquad \forall i\equiv \dim M\mod 2, 
\]
where $R=\mathcal{P}^{-1}\ZZ$ is an extension $\ZZ$ where we allow to divide by the primes in a finite set $\mathcal{P}$.
We prove this by first showing it for CROSSes, using the energy function $E:\lp M\to \RR$ 
of their canonical metric as an $\sphere^1$-equivariant Morse-Bott function, and then proving the result in general.

Since the canonical metrics on CROSSes are Besse metrics
and
all geodesics have the same period $\pi$,
the critical energies are $e_k={1\over 2}k^2\pi^2$ for $k\geq0$.

The critical set $C_0$ consists of constant curves, and thus it is homeomorphic to $M$. The set $C_1$ consists 
of all the simple geodesics
in $M$ and, for $k>1$, the set $C^k=E^{-1}(e_k)$ consists of the $k$-iterates of the geodesics in $C^1$.
Since all the geodesics of $M$ are closed of the same length,
for every $k\leq1$ and every unit tangent vector $v$ the geodesic $c_v=\exp (ktv)$,
$t\in [0,2\pi]$ is a geodesic in $C^k$, and the map sending $v$ to $c_v$ is a homeomorphism $T^1M\to C^k$.

Let $i_k$ denote the index of $C^k$. For every $a\in \RR$ we let $\lp^{k}\In \lp M$ denote
\[
\Lambda^k=E^{-1}([0,e_k+\epsilon))\simeq E^{-1}([0,e_{k+1}-\epsilon))
\]
for some $\epsilon$ sufficiently small.

Since the negative bundle of every critical submanifold is orientable, we have
\begin{align}
H^*_{\sphere^1}(\lp^{0};\ZZ)&=H^*_{\sphere^1}(M;\ZZ),\\
H^*_{\sphere^1}(\lp^{k}, \lp^{k-1};\ZZ)&=H^{*-i_k}_{\sphere^1}(C^k;\ZZ).\label{E:induction}
\end{align}
The action of $\sphere^1$ on $\lp M$ reduces to a trivial action on $C^0$, a free action on $C^1$ and, for $k>1$, an almost free action on $C^k$ with ineffective kernel $\ZZ_k$. Moreover, if $k>1$ then the action of $\sphere^1/\ZZ_k$ on $C^k$ becomes free. In particular,
\begin{align}
H^*_{\sphere^1}(C^0;\ZZ)&=H^*(M\times B\sphere^1;\ZZ)\\
H^*_{\sphere^1}(C^1;\ZZ)&=H^*(T^1M/\sphere^1;\ZZ)\\
H^*_{\sphere^1}(C^k;\ZZ)&=H^*(T^1M/\sphere^1\times B\ZZ_k;\ZZ),\qquad \forall k>1.
\end{align}


\begin{prop}\label{P:integral-cohom}
Given $M$ a CROSS, then the group $H^*(T^1M/\sphere^1;\ZZ)$ is isomorphic
to $H^*(N;\ZZ)$ where $N$ is given in the following Table:

\begingroup
\renewcommand{\arraystretch}{1.2}
\[
\begin{array}{|c||c |c |c |c |c|}
\hline
M	& \sphere^{2m}		& \sphere^{2m+1}				& \CC\PP^{m}					& \HH\PP^{m}						&  Ca\PP^2				\\ \hline
N	& \CC\PP^{2m-1}	&\sphere^{2m}\times\CC\PP^{m}	& \CC\PP^{m-1}\times \CC\PP^m	& \HH\PP^{m-1}\times \CC\PP^{2m+1}	& \sphere^8\times \CC\PP^{11} \\ \hline
\end{array}
\]
\endgroup
\end{prop}

\begin{proof}
In the Serre spectral sequence of $\sphere^{n-1}\to T^1M\to M$ the only nonzero differential is the transgression map $\ZZ=H^{n-1}(\sphere^{n-1};\ZZ)\to H^{n}(M;\ZZ)=\ZZ$ which is the multiplication by $\chi(M)$. It is thus is easy to compute the integral cohomology groups of $T^1M$ (only the nontrivial groups are mentioned):
\begin{align*}
M=\sphere^n,\, n\textrm{ even }&\quad H^q(T^1M;\ZZ)=\left\{\begin{array}{ll}\ZZ& q=0,2n-1\\ \ZZ_2& q=n\end {array}\right.\\
M=\sphere^n,\, n\textrm{ odd }&\quad H^q(T^1M;\ZZ)=\left\{\begin{array}{ll}\ZZ& q=0,n-1,n,2n-1 \end {array}\right.
\end{align*}
\begin{align*}
M=\CC\PP^{m},\, n=2m&\quad H^q(T^1M;\ZZ)=\left\{\begin{array}{ll}\ZZ& q=2j\quad \,j=0,\ldots m-1\\ &q=2(m+j)+1\\ \ZZ_{n+1}& q=2m\end {array}\right.\\
M=\HH\PP^{m},\, n=4m&\quad H^q(T^1M;\ZZ)=\left\{\begin{array}{ll}\ZZ& q=4j\quad \,j=0,\ldots m-1\\ &q=4(m+j)+3\\ \ZZ_{n+1}& q=4m\end {array}\right.\\
M=Ca\PP^2,\, n=16&\quad H^q(T^1M;\ZZ)=\left\{\begin{array}{ll}\ZZ& q=0,8,23,31\\ \ZZ_3& q=16\end {array}\right.
\end{align*}
The result follows by analysing the Gysin sequence of the principal bundle $\sphere^1\to T^1M\to (T^1M)/\sphere^1$.
We show the explicit computations for the case of $M=\HH\PP^m$.
\\

It is enough to show that for $q\leq \dim\big(T^1\HH\PP^m/\sphere^1\big)/2=4m-1$, the cohomology of $T^1\HH\PP^m/\sphere^1$ is
\[
H^q(T^1\HH\PP^m/\sphere^1;\ZZ)=\left\{\begin{array}{ll}\ZZ^{j+1}& q=4j,4j+2\\ 0 & \textrm{otherwise}\end{array}\right.
\]
which coincides with the cohomology of $\HH\PP^{m-1}\times \CC\PP^{2m+1}$ in that range. Since $T^1\HH\PP^m/\sphere^1$ and $\HH\PP^m\times \CC\PP^{2m+1}$ have the same dimension and satisfy Poincar\'e duality, the isomorphism in cohomology follows for all $q$.

Recall that, given a principal bundle $\sphere^1\to E\to B$, the Gysin sequence reads
\begin{equation}\label{E:Gysin}
\ldots \lra H^{q-1}(B;\ZZ)\stackrel{\cup e}{\lra} H^{q+1}(B;\ZZ) \lra H^{q+1}(E;\ZZ) \lra\ldots
\end{equation}
where $e\in H^2(B;\ZZ)$ is the Euler class of the bundle.

For the sake of notation, we will denote $E=T^1\HH\PP^m$ and $B=T^1\HH\PP^m/\sphere^1$ for the rest of the proof. Notice that for If $q\leq4m-1$, $H^q(E;\ZZ)=0$ unless $q=4j$, and by \eqref{E:Gysin} it follows that $H^{4q}(B;\ZZ)=H^{4q+2}(B;\ZZ)$ and $H^{4q+1}(B;\ZZ)=H^{4q+3}(B;\ZZ)$. In particular, $H^0(B;\ZZ)=H^2(B;\ZZ)=\ZZ$. Again from \eqref{E:Gysin}, $0\to H^1(B;\ZZ)\to H^1(E;\ZZ)=0$ and therefore $H^1(B;\ZZ)=H^3(B;\ZZ)=0$, thus proving the claim for $q< 4$. 

Suppose by induction that the claim is true for  $k<4j$, for $j<m$ (if it were $j=m$ we would be done). Again by \eqref{E:Gysin} we have
\[
0=H^{4j-1}(B;\ZZ)\to H^{4j+1}(B;\ZZ)\to H^{4j+1}(E;\ZZ)=0
\]
and thus $H^{4j+1}(B;\ZZ)=H^{4j+3}(B;\ZZ)=0$. Moreover,
\[
0\ra H^{4j-2}(B;\ZZ)\stackrel{\cup e}{\ra} H^{4j}(B;\ZZ) \ra H^{4j}(E;\ZZ)\ra 0
\]
Since $j<m$, then $H^{4j}(E;\ZZ)=\ZZ$ and therefore $H^{4j}(B;\ZZ)=H^{4j-2}(B;\ZZ)\oplus \ZZ$. Since $H^{4j+2}(B;\ZZ)=H^{4j}(B;\ZZ)$, this proves the induction step.

\end{proof}


\begin{cor}\label{C:eqcohCROSS}
Let $M$ be a CROSS of dimension $n$. The relative equivariant cohomology $H^*_{\sphere^1}(\lp M, M;\ZZ)$ satisfies
\[
\left\{\begin{array}{ll}
H^{odd}_{\sphere^1}(\lp M,M;\ZZ)=0&\textrm{if $n$ is odd}\\
H^{ev}_{\sphere^1}(\lp M,M;\ZZ)=0&\textrm{if $n$ is even}
\end{array}
\right.
\]
\end{cor}
\begin{proof}
It follows from Proposition \ref{P:integral-cohom} that $H^{odd}(T^1M/\sphere^1;\ZZ)=0$ for all cases. Recall that for any $k>0$, $H^*_{\sphere^1}(C^k;\ZZ)=H^*(T^1M/\sphere^1\times B\ZZ_k;\ZZ)$ and thus, since $H^{odd}(B\ZZ_k;\ZZ)=0$ as well, it follows that $H^{odd}_{\sphere^1}(C^k;\ZZ)=0$ for all $k>0$. Moreover, by \cite{Wil} the index $i_k$ is even if and only if $n$ is odd for any $k=0$. By equation \eqref{E:induction}, it follows that for every $k>0$,
\begin{equation}\label{E:ind-step}
\left\{\begin{array}{ll}
H^{odd}_{\sphere^1}(\lp^{k},\lp^{k-1};\ZZ)=0&\textrm{if $n$ is odd}\\
H^{ev}_{\sphere^1}(\lp^{k},\lp^{k-1};\ZZ)=0&\textrm{if $n$ is even}
\end{array}\right.
\end{equation}
Using the long exact sequence in cohomology
\[
\ldots \to H^q_{\sphere^1}(\lp^{k_1},M;\ZZ)\to H^q_{\sphere^1}(\lp^{k_2},M;\ZZ)\to  H^{q+1}_{\sphere^1}(\lp^{k_1},\lp^{k_2};\ZZ)\to \ldots
\]
for any $k_1>k_2$, we obtain by induction that for every $k$
\[
\left\{\begin{array}{ll}
H^{odd}_{\sphere^1}(\lp^{k},M;\ZZ)=0&\textrm{if $n$ is odd}\\
H^{even}_{\sphere^1}(\lp^{k},M;\ZZ)=0&\textrm{if $n$ is even}
\end{array}
\right.
\]
Taking the direct limit as $k\to \infty$ we obtain the result.

\end{proof}

\begin{rem}
By the Universal Coefficient Theorem, Corollary \ref{C:eqcohCROSS} also holds with coefficients in $R=\mathcal{P}^{-1}\ZZ$.
\end{rem}

\subsection{Integral cohomology CROSSes}
Let now $M$ be a  manifold whose integral cohomology is that of a CROSS.
The goal of this section is to prove the following generalization of Corollary \ref{C:eqcohCROSS}.

\begin{prop}\label{P:Same-up-to-primes}
Let $M$ be a compact manifold whose rational cohomology ring is isomorphic to that of a CROSS $M'$.
Then there is a ring isomorphism
$$\phi:H^*_{\sphere^1}(\lp M,M;R)\to H^*_{\sphere^1}(\lp M',M';R),$$ 
where $R=\mathcal{P}^{-1}\ZZ$ for a suitably chosen finite collection of primes $\mathcal{P}$.
\end{prop}
\begin{proof}
Since $M$ has the rational cohomology of a CROSS $M'$ then it is \emph{formal}, i.e. its rational homotopy type only depends on the rational cohomology ring (cf. for example \cite[Cor. 2.7.9]{Allday-Puppe}). Therefore there is a space $M_0$ and maps $M\to M_0\leftarrow M'$ that induce isomorphisms in rational cohomology. Since the cohomology groups of $M$ and $M'$ are 
finitely generated, it follows that there is a finite set of primes $\mathcal{P}$ and a space $M_{\mathcal{P}}$ (called \emph{localisation} of $M$ at $\mathcal{P}$) with maps $M\to M_{\mathcal{P}}\leftarrow M'$, which induce isomorphism in cohomology with coefficients in the localised ring $R=\mathcal{P}^{-1}\ZZ$ (cf. for example \cite[Cor. 5.4(c)]{HMR}). By Corollary 4.4 of  \cite{CLOT} there is a homotopy commutative diagram
\[
\xymatrix{
\lp M \ar[d] \ar^{\overline{\varphi}}[r] & \lp M_{\mathcal{P}} \ar[d] & \lp M' \ar_{\ol{\varphi}'}[l] \ar[d]\\
 M \ar^{\overline{\varphi}}[r] &  M_{\mathcal{P}} &  M' \ar_{\varphi'}[l]}
\]
with horizontal arrows inducing isomorphism in $H^*(\cdot; R)$, where  the map $\ol{\varphi}:[\sphere^1,M]\to \lp M_{\mathcal{P}}=[\sphere^1, M_{\mathcal{P}}]$ takes a curve $c$ to $\varphi\circ c$, and similarly for $\ol{\varphi}'$. In particular, $\ol{\varphi}$ and $\ol{\varphi}'$ are $\sphere^1$ equivariant, and therefore they induce isomorphisms in relative equivariant cohomology
\[
H^*_{\sphere^1}(\lp M, M;R)\stackrel{\ol{\varphi}^*}{\longleftarrow}H^*_{\sphere^1}(\lp M_{\mathcal{P}}, M_{\mathcal{P}};R)\stackrel{\ol{\varphi}'^*}{\lra} H^*_{\sphere^1}(\lp M',M';R)
\]
The composition $\phi=\ol{\varphi}'^*\circ(\ol{\varphi}^*)^{-1}$ is the isomorphism we wanted.
\end{proof}

\begin{cor}\label{C:right-parity-R}
Suppose the manifold $M$ is a rational cohomology CROSS of dimension $n$. Then there is a finite set $\mathcal{P}$ of primes, such that the relative equivariant cohomology $H^*_{\sphere^1}(\lp M, M; R)$, $R=\mathcal{P}^{-1}\ZZ$, satisfies
\[
\left\{\begin{array}{ll}
H^{odd}_{\sphere^1}(\lp M,M; R)=0&\textrm{if $n$ is odd}\\
H^{ev}_{\sphere^1}(\lp M,M;R)=0&\textrm{if $n$ is even}
\end{array}
\right.
\]
\end{cor}
\begin{proof} Suppose for sake of simplicity that $n$ is odd, the case $n$ even follows in the same way. By Proposition \ref{P:Same-up-to-primes}, it is enough to check that the theorem holds for $M$ a CROSS. In this case, by Corollary \ref{C:eqcohCROSS} we have that $H^{odd}_{\sphere^1}(\lp M, M;\ZZ)=0$ and, by the Universal Coefficient Theorem,
we have that $H_{odd}((\lp M)_{\sphere^1}, M_{\sphere^1};\ZZ)$ is torsion and $H_{even}((\lp M)_{\sphere^1},M_{\sphere^1};\ZZ)$ is free. Therefore, $\textrm{Hom}(H_{odd}((\lp M)_{\sphere^1}, M_{\sphere^1};\ZZ), R)=0$ since $R$ is torsion free, and $\textrm{Ext}(H_{even}((\lp M)_{\sphere^1},M_{\sphere^1};\ZZ),R)=0$. Again by the Universal Coefficient Theorem, $H^{odd}_{\sphere^1}(\lp M, M;R)=0$.
\end{proof}


\section{Index gap}
\label{S:index-gap}

Let $(M,g)$ be a Besse manifold. 
For a geodesic $c:\sphere^1\to M^n$ let $c^q$ denote the $q$-iterate of $c$. Recall that by Wadsley theorem,
there is a number $L$ such that every prime geodesic has length equal to $L/k$ for some integer $k$.

\begin{defn}
A closed geodesic $c$ in $M$ is called \emph{regular} if its length is a multiple of $L$. Moreover,
a critical set $C\In \lp M$ for the energy functional is called \emph{regular} if it contains regular geodesics.
\end{defn}

For every primitive closed geodesic $c$ there is some $q$ such that $c^q$ is regular.
It follows in particular that any regular set $C$, containing geodesics of length, say, $kL$,
is homeomorphic to the unit tangent bundle $T^1M$, via the map $T^1M\to C$ sending $(p,v)$ to $c(t)=\exp_p(t\,kv)$.
\\

Recall that the index $\ind(c)$ of a closed geodesic is the index of the Hessian of the Energy functional 
$E:\lp M\to \RR$, at $c$. Along a critical manifold $C$ the index of the hessian remains constant,
so sometimes we will also refer to the index $\ind(C)$. Similarly, the \emph{extended index} is given by
$\ind_0(c)=\ind(c)+\textrm{null}(c)$, where $\textrm{null}(c)$ is the dimension of the kernel of the Hessian
of $E$ at $c$. Notice that $\null(c)$ equals the number of periodic Jacobi fields along $c$, which
equals the dimension of the subspace of $V\oplus V$, $V=\scal{c'(0)}^{\perp}$, fixed by the Poincar\'e map of $c$.
%
%

\begin{prop}[Index Gap]\label{P:index-gap}
Let $c:\sphere^1\to M^n$ be a geodesic such that $c^q$ is regular. Then, for any $l$, $0<l< q$:
\begin{subequations}
\begin{align}
\ind(c^{q+l})&= \ind(c^{q})+\ind(c^l)+(n-1) \label{E:a}\\
\ind_0(c^q)&= \ind_0(c^{q-l})+\ind(c^l)+(n-1) \label{E:b}
\end{align}
\end{subequations}
\end{prop}

\begin{proof}
Recall, from example from \cite{BTZ}, that the index  and extended index of a geodesic $c:[0,2\pi]\to M$, with Poincar\'e map $P$ can be computed as 
\begin{align*}
\ind (c)&=\ind_\Omega(c)+(\ind+ \dim\ker)\tilde{H}- \ker(P-\Id)\\
\ind_0 (c)&=\ind_\Omega(c)+(\ind+ \dim\ker)\tilde{H}
\end{align*}
where $\ind_\Omega c$ is the sum of conjugate points of $c(0)$ along $c$, and $\tilde{H}$ is the \emph{concavity form} defined on $(P-\Id)^{-1}(0\oplus V)$ as
\[
\tilde{H}(X,Y)=-\omega((P-\Id)X,Y).
\]
As the summand $(\ind+ \dim\ker)\tilde{H}- \ker(P-\Id)$ only depends on $P$, we will call this $\ind_P(c)$.

To prove Equation \eqref{E:a} it is enough to prove that
\begin{align*}
\ind_{\Omega}(c^{q+l})&=\ind_\Omega(c^q)+\ind_\Omega(c^l)+(n-1)\\
\ind_P(c^{l+q})&=\ind_P(c^q)+\ind_P(c^l).
\end{align*}

The first equation holds because for every conjugate point $t_0$ of $c^{q+l}$, we have either have $t_0\in (0,q)$, $t_0=q$, or $t_0\in (q,q+l)$. By definition there are exactly $\ind_{\Omega}(c^q)$ many conjugate points of the first type. Since every Jacobi field $J$ with $J(0)=0$ also satisfies $J(q)=0$, we have in particular that $t_0=q$ has multiplicity $n-1$. Finally, since every Jacobi field is periodic on $[0,q]$, a Jacobi field $J$ on $[0,q+l]$ satisfies $J(0)=J(t_0)=0$ for some $t_0\in (q,q-l)$ if and only if $K(t)=J|_{[q,q+l]}(t-q)$ satisfies $K(0)=K(t_0-t)=0$, thus there are exactly $\ind_\Omega(c^l)$ many conjugate points of the last type.

For the second equation, just notice that since $c^q$ has Poincar\'e map $\Id$, it follows that $\ind_P(c^q)=0$, and since $c^l$ and $c^{q+l}$ have the same Poincar\'e map, il follows that $\ind_P(c^l)=\ind_P(c^{q+l})$.
\\

To prove Equation \eqref{E:b}, we first prove a couple of easy lemmas.

\begin{lem}\label{L:lagr-symp}
Let $\mathcal{L}$ be a Lagrangian subspace of $(V\oplus V,\omega)$ and let $K$ be a symplectic subspace. Then
\[
\dim (\mathcal{L}\cap K^\perp) -\dim(\mathcal{L}\cap K)+\dim K=\dim V.
\]
where $K^\perp=\{x\in V\oplus V\mid \omega(x,K)=0\}$.
\end{lem}
\begin{proof}
We compute
\begin{align*}
\dim (\mathcal{L}\cap K^\perp)	&=\dim(V\oplus V)-\dim (\mathcal{L}\cap K^\perp)^\perp\\
						&=2\dim(V)-\dim(\mathcal{L}^\perp+ K)\\
						&=2\dim(V)-\dim (\mathcal{L}^\perp) -\dim(K)+\dim(\mathcal{L}\cap K)\\
						&=\dim(V)-\dim(K)+\dim(\mathcal{L}\cap K).
\end{align*}
\end{proof}

\begin{lem}\label{L:dim-preimage}
Let $K=\ker(P-\Id)\In V\oplus V$, where $P\in \U(n-1)\In \Sp(n-1,\RR)$. Then for any subspace $\mathcal{L}$ of $V\oplus V$,
\[
\dim (P-\Id)^{-1}(\mathcal{L})=\dim K+\dim (\mathcal{L}\cap K^\perp).
\]
\end{lem}
\begin{proof}
Notice first that $K$ is the (generalised) eigenspace of $P$ with eigenvalue 1, and therefore it is a symplectic subspace of $V\oplus V$ (cf. for example \cite{BTZ}, p. 220-222). Because $P$ lies in the maximal compact subgroup of $\Sp(n-1,\RR)$, it is possible to write $P=\left(\begin{array}{cc}\Id_K &  \\ & P|_{K^\perp}\end{array}\right)$. In particular, $\textrm{Im}(P-\Id)\In K^\perp$, $K^\perp$ is $(P-\Id)$-invariant, and the restriction $(P-\Id)|_{K^\perp}$ is invertible. Therefore
\begin{align*}
\dim(P-\Id)^{-1}(\mathcal{L})&=\dim(P-\Id)^{-1}(\mathcal{L}\cap K^{\perp})\\
						&=\dim(K)+ \dim(P-\Id)|_{K^\perp}^{-1}(\mathcal{L}\cap K^{\perp})\\
						&=\dim(K)+ \dim(\mathcal{L}\cap K^{\perp}).
\end{align*}
\end{proof}

We can now prove Equation \eqref{E:b}. Let $P$ be the Poincar\'e map of $c$, so that $c^l$ and $c^{q-l}$ have Poincar\'e maps $P^l$ and $P^{q-l}$, respectively, and $P^q=\Id$.  Let us denote the Lagrangian subspace $0\oplus V$ with $\mathcal{L}$, and the symplectic space $\ker(P^l-\Id)=\ker(P^{q-l}-\Id)$ with $K$. By definition, $\dim K=\null(c^l)=\null(c^{q-l})$.

Since $P$ satisfies $P^q=\Id$, we have $\ind_P(c^q)=0$, $\null(c^q)=2(n-1)$ and therefore $\ind_0(c^q)=\ind_\Omega(c^q)+2(n-1)$. Equation \eqref{E:b} can be then simplified as
\[
\ind_0(c^{q-l})+\ind(c^l) =\ind_{\Omega}(c^q)+(n-1)
\]
To prove the equation above it is enough to prove that
\begin{align*}
\ind_{\Omega}(c^{q-l})+\ind_\Omega(c^l)&=\ind_\Omega(c^q)-\mu(q-l)\\
\ind_P(c^{q-l})+\ind_P(c^q)+\null(c^{q-l})&=(n-1)+\mu(q-l)
\end{align*}
where $\mu(q-l)$ denotes the number multiplicity of $q-l$ as a conjugate point of $c(0)$.

The first equation holds because for every conjugate point $t_0$ of $c^{q}$, we have either have $t_0\in (0,q-l)$, $t_0=q-l$, or $t_0\in (q-l,q)$, and by definition there are exactly $\ind_{\Omega}(c^{q-l})$ many conjugate points of the first type, and $\ind_\Omega(c^l)$ many conjugate points of the last type.

For the second equation notice that, since the Poincar\'e maps of $c^l$ and $c^{q-l}$ are inverses of each other, their concavity forms $\tilde{H}_{1}$ and $\tilde{H}_{2}$ (defined on the same space $(P^{l}-\Id)^{-1}(\mathcal{L})=(P^{q-l}-\Id)^{-1}(\mathcal{L})$) satisfy the relation $\tilde{H}_1=-\tilde{H}_{2}$. Therefore
\begin{align*}
(\ind+\dim \ker)\tilde{H}_{1}+(\ind+\dim \ker)\tilde{H}_{2}=\dim (P^{l}-\Id)^{-1}(\mathcal{L}) +\dim\ker \tilde{H}_{2}.
\end{align*}
By Lemma \ref{L:dim-preimage}, $\dim (P^{l}-\Id)^{-1}(\mathcal{L})=\dim(K)+ \dim(\mathcal{L}\cap K^{\perp})$, and the equation after (1.3) in \cite{BTZ} gives
\[
\dim\ker\tilde{H}_{2}=\dim K +\mu(q-l)-\dim(\mathcal{L}\cap K).
\]
Putting these equations together, we compute $\ind_P(c^{q-l})+\ind_P(c^q)+\null(c^{q-l})$ as
\begin{align*}
&(\ind+\dim \ker)\tilde{H}_{1}+(\ind+\dim \ker)\tilde{H}_{2}-\null(c^l)\\
								=&\dim (P^{l}-\Id)^{-1}(\mathcal{L}) +\dim\ker \tilde{H}_{2}-\dim K\\
								=&\mu(q-l)+\dim(K)+ \dim(\mathcal{L}\cap K^{\perp})-\dim(\mathcal{L}\cap K)\\
								=&\mu(q-l) +(n-1)
\end{align*}
where in the last equality we used Lemma \ref{L:lagr-symp}. Thus Equation \eqref{E:b} holds.

\end{proof}

Given a  Besse manifold $M$, let $i(M)$ denote the minimal index 
of a critical set for the energy functional in $\lp M$. It is easy to see that $i(M)$ is the lowest index $q$ 
such that $H^q_{\sphere^1}(\lp M,M;\QQ)\neq 0$. We thus have the following values for CROSSes:
\[
i(\sphere^n)=n-1,\qquad i(\CC\PP^n)=1,\qquad i(\HH\PP^n)=3,\qquad i(Ca\PP^2)=7.
\]
In general, since $M$ is a simply connected rational cohomology CROSS then by Proposition~\ref{P:Same-up-to-primes} $i(M)$
only depends on the CROSS $M$ is modelled on.

The following is a straightforward consequence of Proposition~\ref{P:index-gap}.
\begin{cor}\label{C:cor-remark}
Given a rational cohomology CROSS $M$ and a geodesic $c$ such that $c^q$ is regular, we have
\begin{align}
\ind(c^{k})&\geq \ind(c^{q})+(n-1)+i(M) \qquad\textrm{ if }k>q\label{eq1}\\
\ind_0(c^{k})&\leq \ind(c^{q})+(n-1)-i(M) \qquad\textrm{ if }k<q. \label{eq3}
\end{align}
Moreover, the inequality in \eqref{eq1} (resp. the inequality in \eqref{eq3}) is strict unless $k=q+1$ (resp. $k=q-1$) and $\ind(c)=i(M)$.
\end{cor}

\section{Perfectness of the energy functional}
\label{S:perfectness}
The goal of this section is to prove Theorem \ref{T:perfectness}, that is, that for every simply connected
Besse manifold $M$, the energy functional $E:\lp M\to \RR$ is perfect with respect
to the $\sphere^1$ equivariant, rational cohomology of $(\lp M,M)$.
\\

In Section \ref{SS:orientable-case} we prove Theorem \ref{T:perfectness} in the special case in which
all negative bundles are orientable. As pointed out in Corollary \ref{C:orientability-spin}, this is the case of spin manifolds, like manifolds with the integral cohomology ring of spheres, quaternionic projective spaces, or the Cayley plane. Finally, in \ref{SS:non-orientable-case} we prove Theorem  \ref{T:perfectness} in the general case.

\subsection{When the negative bundles are all orientable}\label{SS:orientable-case}

Let $M$ be a Besse manifold, and let $C_1,\ldots C_k\In \lp M$ be 
the critical sets of $E$ containing prime geodesics. For every $C\in\{C_1,\ldots C_k\}$, and every $q\in \ZZ$ 
let $C^q$ denote the critical set consisting of $q$-iterates of geodesics in $C$. Clearly,
every critical set of $E$ is of the form $C^q$ for some $q\in \ZZ$ and $C\in \{C_1,\ldots C_k\}$.


The core result of the section is Proposition \ref{P:no-odd-degree}, where we prove that the rational, 
$\sphere^1$-equivariant cohomology of every critical set is concentrated in even degrees.
This fact, together with the index parity result of the index in \cite{Wil}, will allow us to prove
Theorem \ref{T:perfectness} by the lacunarity principle.

Before that, however, we first need to analyse the structure of $p$-torsion in the cohomology of the critical 
sets, for big primes $p$.

\begin{rem}\label{R:notation}
\emph{For the rest of this section, we will denote by $R$ the localisation ring $R=\mathcal{P}^{-1}\ZZ$ where $\mathcal{P}$ is a finite collection of primes such that 
\begin{itemize}
\item every prime dividing the order of some element in $H^*(C;\ZZ)$ and $H^*_{\sphere^1}(C;\ZZ)$, $C\in \{C_1,\ldots C_k\}$ is contained in $\mathcal{P}$.
\item Proposition \ref{P:Same-up-to-primes} and Corollary \ref{C:right-parity-R} hold for $R=\mathcal{P}^{-1}\ZZ$.
\item For every prime geodesics of length $L/k$, all the prime divisors of $k$ are contained in $\mathcal{P}$. Equivalently, for any $p\notin \mathcal{P}$ and any critical set $K$, the set $K^p$ is also a critical set.
\end{itemize}}
\end{rem}

For this to make sense, one must first make sure that there are only finitely many primes dividing 
the orders of the elements of $H^*(C;\ZZ)$ and $H^*_{\sphere^1}(C;\ZZ)$. However, this is clearly 
true for $H^*(C;\ZZ)$ because it is finitely generated and in particular contains finitely many torsion elements.
For $H^*_{\sphere^1}(C;\ZZ)=H^*(C_{\sphere^1};\ZZ)$. From the Gysin sequence of 
the $\sphere^1$- bundle $C\to C_{\sphere^1}$ we obtain that there is an 
isomorphism $H^i_{\sphere^1}(C;\ZZ)\to H^{i+2}_{\sphere^1}(C;\ZZ)$ for any $i>\dim(C)$ and therefore the torsion of
$H^*_{\sphere^1}(C;\ZZ)$ is the same as the torsion of $H^{\leq \dim (C)}_{\sphere^1}(C;\ZZ)$, which is finite.

We recall the following basic fact
\begin{lem}\label{L:rel-contribution-degrees}
For any critical energy $e_r$ with critical manifold $K_r$, the group $H^i_{\sphere^1}(\lp^r,\lp^{r-1};\ZZ)$ can contain torsion free elements only in degrees $i\in \{\ind(K_r),\ldots, \ind_0(K_r)\}$.
\end{lem}
\begin{proof}
This is equivalent to showing that $H^i_{\sphere^1}(\lp^r,\lp^{r-1};\QQ)=0$ for $i\notin \{\ind(K_r),\ldots, \ind_0(K_r)\}$. Recall that $H^i_{\sphere^1}(\lp^r,\lp^{r-1};\QQ)=H^{i-\ind(K_r)}_{\sphere^1}(K_r;\QQ)$. Moreover, since $\sphere^1$ acts almost freely on $K_r$, the quotient $K_r/\sphere^1$ is an orbifold of dimension $\dim(K_r)-1=\null(K_r)$, and therefore $H^{i-\ind(K_r)}_{\sphere^1}(K_r;\QQ)=H^{i-\ind(K_r)}(K_r/\sphere^1;\QQ)$. This is clearly $0$ for $i\notin \{0,\ldots \null(K_r)\}$ and this proves the result.
\end{proof}

\begin{prop}\label{P:H-odd-1} \label{P:H-odd-2}
Let $K$ be a critical set of the energy functional, and let $p\notin \mathcal{P}$ be a prime. Then:
\begin{enumerate}
\item If $K$ is a critical manifold with $H^{odd}_{\sphere^1}(K;\QQ)= 0$, then any critical manifold of the form $K^q$ has no $p$-torsion in $H^{odd}_{\sphere^1}(K^q;\ZZ)$.
\item If $K$ is a critical manifold with $H^{odd}_{\sphere^1}(K;\QQ)\neq 0$, then $H^{2h+1}_{\sphere^1}(K^p;\ZZ)$ contains $p$-torsion for every $2h+1\geq \dim(K)-1$.
\end{enumerate}
\end{prop}
\begin{proof}
Recall that for every coefficient ring, $H^*_{\sphere^1}(K^{q})=H^*(K^{q}_{\sphere^1})$ with
\[
K^{q}_{\sphere^1}\simeq K\times_{\sphere^1}E\sphere^1
\]
where $\sphere^1$ acts on $K\times E\sphere^1$ by $z\cdot (x,a)=(z^q\cdot x, z \cdot a)$.
We can rewrite this slightly differently but homotopically equivalent as
\[
K^{q}_{\sphere^1}\simeq K\times_{\sphere^1}E\sphere^1\times E\sphere^1
\]
with $z (x,a,b)=(z^qx,z^qa,zb)$. Since the above $\sphere^1$-action extends naturally to a free torus action, 
 we see that $K^{q}_{\sphere^1}$ is an $\sphere^1$-bundle 
over $K_{\sphere^1}\times B\sphere^1$.

The Euler class of this bundle $\xi\colon K^{q}_{\sphere^1}\rightarrow K_{\sphere^1}\times B\sphere^1$
is given by
\[
\ol{e}=e\otimes 1-q(1\otimes c)\in H^2(K_{\sphere^1}\times B\sphere^1;R)\subset  H^*_{\sphere^1}(K;R)\otimes H^*(B\sphere^1;R).
\]
where $c$ is the generator of $H^2(B\sphere^1;R)$ and $e$ is the Euler class of the bundle 
$K\rightarrow K_{\sphere^1}$.
The Gysin sequence for the bundle $K^q_{\sphere^1}\to K_{\sphere^1}\times B\sphere^1$ reads
\begin{align}\label{E:Gysin}
H^{2k+1}(K_{\sphere^1}\times B\sphere^1;R)\stackrel{\xi^*}{\lra} H^{2k+1}(K_{\sphere^1}^q;R)\stackrel{\xi_!}{\lra} H^{2k}(K_{\sphere^1}\times B\sphere^1;R)
\end{align}
If $H^{2k+1}(K_{\sphere^1}^q;R)$ had some $p$-torsion element $x$, 
it would lie in the kernel of $\xi_!$ because $H^{*}(K_{\sphere^1}\times B\sphere^1;R)$ 
does not have any $p$-torsion by definition of $\mathcal{P}$. Then it would be $x=\xi^*(y)$
for some $y\in H^{2k+1}(K_{\sphere^1}\times B\sphere^1;R)$. By the choice of $p$, $y$ cannot be $p$-torsion and
thus it must be torsion free, which implies that $H^{odd}(K_{\sphere^1};\QQ)\neq 0$. This proves the first point.

Suppose now that $H^{odd}_{\sphere^1}(K;\QQ)\neq 0$, and let $x\in H^{h_0}(C_{\sphere^1};R)$, $h_0$ odd, 
be a torsion-free element not divisible by $p$, such that $x\cup e=0$. Such an $x$ exists because 
$H^{odd}(K_{\sphere^1};\QQ)$ is nonzero and, since the $\sphere^1$-action on $K$ is almost free, 
the cohomological dimension of $H^*(K_{\sphere^1};\QQ)$ is at most $\dim(K) -1$. The Gysin sequence of $\xi$
reads
\[
 H^{h_0+2m-2}(K_{\sphere^1}\times B\sphere^1;R)\stackrel{\cup \ol{e}}{\lra}
 H^{h_0+2m}(K_{\sphere^1}\times B\sphere^1;R) \stackrel{\xi^*}{\lra} H^{h_0+2m}(K_{\sphere^1}^{ p};R)
\]
The map $\cup \ol{e}$ is easily seen to be injective. Combining with the fact 
that the primitive element $(x\otimes c^{m-1})$ is mapped to
$p\cdot(x\otimes c^m)$, 
we deduce that $p$-torsion can be found in the image of 
$\xi^*$.
%
\end{proof}

Recall from Section \ref{S:recall} that we denote by $K_1, K_2,\ldots$ the list of the critical sets of $E$ of positive energy, in increasing order $e_1<e_2<\ldots$. For every $k\geq 0$, let $i_k$ denote the index of $K_k$. Moreover, let us define the sublevel sets $\lp^k=E^{-1}([0, e_k+\epsilon))\In \Lambda M$ for some $\epsilon>0$ small enough.
\\

\begin{prop}\label{P:no-odd-degree}
For every critical set $K$ of the energy functional, one has $H^{odd}_{\sphere^1}(K;\QQ)=0$.
\end{prop}

\begin{proof} We choose a large prime $p$ and consider the ring 
\[S:=\ZZ\bigl[\{1/q
\mid \mbox{$q$ prime $q\neq p$}\}\bigr]\]
all cohomology groups in this proof will be with respect to coefficients in $S$. 
Equivalently one can work with integral coefficients and ignore all torsion except for $p$-torsion.

We argue by contradiction and consider the smallest length $l$ for which 
the set $C=C_i$ of geodesics of length $l$ has nontrivial rational cohomology in some 
odd degree.  
We let $C_1,\ldots, C_{i-1}$ be the critical manifolds of smaller length 
containing primitive geodesics and 
$C_{i+1},\ldots,C_m\cong T^1M$ the ones of larger length containing primitive geodesics. 
Choose $\epsilon>0$ such that there are no closed geodesics of
length $l'\in ([l-\epsilon,l+\epsilon]\setminus \{l\})$.

We may assume that $p$ is so large that we can find positive integers $u_1<u_2$ such 
that $p(l-\epsilon)<u_1L<pl<u_2L<p(l+\epsilon)$, where $L$ is the common period of all unit speed geodesics.

Recall that any critical manifold is given by iterating the geodesics in $\{C_1,\ldots C_{m}\}$, that 
is $K_h=(C_{j(h)})^{l(h)}$. 
Let $r$ be the index corresponding to the critical manifold $C_i^p=K_r$.
Furthermore let $r_1<r$ and $r_2>r$ denote the index corresponding to length $u_1L$ and $u_2L$, 
respectively. 

By construction $K_h$ does not contain a $p$-times iterated geodesic if 
$h=r_1,\ldots,r-1$ or $h=r+1,\ldots,r_2$. In the following we denote by $[x]$ the Gau{\ss} bracket of a real number
$x$.\\[2ex]
{\bf Step 1.} Suppose $r_1\le h\le r-1$.
Then the inclusion $\Lambda^h\rightarrow \Lambda M$ induces an epimorphism 
$H^i_{\sphere^1}(\Lambda M,M)\rightarrow H^i_{\sphere^1}(\Lambda^h,M)$ 
for $i\ge \ind(K_{r_2})+n-2$.

Let $e\in H^2(\Lambda M\times_{\sphere^1} E\sphere^1)$ denote the Euler class 
of the $\sphere^1$-bundle $\Lambda M\rightarrow \Lambda M\times_{\sphere^1} E\sphere^1$. 

We first consider the $\Lambda^{r_1-1}$. By the index gap Lemma
we have $\ind_0(K_h)\le \ind(K_{r_1})+(n-1)$, $h=1,\ldots,r_1-1$. 
Hence $H^i(\Lambda^{r_1-1},M)=0$ for $i\ge \ind(K_{r_1})+n$. 
Using the Gysin sequence of the $\sphere^1$-bundle we see that 
$\cup e\colon H^{i}_{\sphere^1}(\Lambda^{r_1-1},M)\rightarrow H^{i+2}_{\sphere^1}(\Lambda^{r_1-1},M)$
is an epimorphism for $i\ge \ind(K_{r_1})+2[(n-1)/2]$.

Furthermore, for each $h<r_1$ we know that either $H^{odd}(K_h)=0$ or $K_h$ does not contain 
a $p$-iterated geodesic. In either case, by Lemma \ref{L:rel-contribution-degrees} and 
Proposition \ref{P:H-odd-1} one has 
$H^i_{\sphere^1}(\Lambda^h,\Lambda^{h-1})=H^{i-\ind(K_h)}_{\sphere^1}(K_h)=0$ if 
$i\equiv n\mod 2$ with $i\ge \ind(K_{r_1})+2[(n-1)/2]$.
Let $i_1=\ind(K_{r_1})$. From the exact sequence of the triple 
$(\Lambda^{r_1},\Lambda^{r_1-1},M)$ we then obtain, for each $i\equiv n+1\mod 2$, the diagram
\begin{eqnarray*}
 H^{i+2-i_1}_{\sphere^1}(K_{r_1}) \to &H^{i+2}_{\sphere^1}(\lp^{r_1},M)\stackrel{i^*}{\lra}& 
H^{i+2}_{\sphere^1}(\lp^{r_1-1},M)\stackrel{d}{\lra} 
0\\
\uparrow\hspace*{4em} &
\uparrow\hspace*{4em} &\hspace*{4em}
\uparrow \\
H^{i-i_1}_{\sphere^1}(K_{r_1}) \to &H^{i}_{\sphere^1}(\lp^{r_1},M)\stackrel{i^*}{\lra}&
H^{i}_{\sphere^1}(\lp^{r_1-1},M)\stackrel{d}{\lra} 
0\\
\end{eqnarray*}
where the vertical arrows are given by cupping with $e$ and the horizontal sequences are exact.
As explained the last vertical map is an epimorphism for $i\ge i_1+2[(n-1)/2]$. 
Since $K_{r_1}=T^1M$ is a regular level the first vertical map is an epimorphism for 
$i\ge i_1+2[(n-1)/2]$ as well. 
By the Four Lemma this implies that the middle map is an epimorphism. 

The Index Gap Lemma implies for any critical manifold $K_j$ with $j> r_1$ that 
$\ind(K_j)\ge j_1=i_1+2[(n-1)/2]+2$. 
Hence the map $H^{j_1}(\Lambda M,M)\rightarrow H^{j_1}(\Lambda^{r_1},M)$ is an isomorphism. 
By the previous discussion we deduce that 
$H^{j}_{\sphere^1}(\Lambda M,M)\rightarrow H^{j}_{\sphere^1}(\Lambda^{r_1},M)$ is surjective for each $j\ge j_1$.

Similarly, the natural map 
$H^{j}_{\sphere^1}(\Lambda^h,M)\rightarrow H^{j}_{\sphere^1}(\Lambda^{r_1},M)$ is an isomorphism
for $j\ge \ind(K_{r_2})+n-2$ and $h=r_1,\ldots, r-1$, since  $K_{r_1+1},\ldots, K_{r-1}$ 
do not contain $p$-iterated geodesics.
Thus the claim of Step 1 follows. \\[2ex]
{\bf Step 2.} $H^j(\Lambda^rM,M)\neq 0$ for all $j\ge \ind(K_{r_2})+n-2$ with 
$j\equiv n \mod 2$.

By Proposition \ref{P:H-odd-1}, $H^{j}_{\sphere^1}(K_r)\neq 0$ for all odd $j\ge \dim(K_r)$. 
Using the exact sequence of the triple $(\lp^r,\lp^{r-1},M)$

\[
H^{j-1}_{\sphere^1}(\lp^{r},M)  \stackrel{\iota^*}{\lra} H^{j-1}_{\sphere^1}(\lp^{r-1},M)\stackrel{d}{\lra} 
H^{j-i_r}_{\sphere^1}(K_r)\stackrel{}{\lra} H^{j}_{\sphere^1}(\lp^{r},M)\to 
\]
The map $\iota^*$ is a factor of $H^{j-1}_{\sphere^1}(\lp,M)  \to H^{j-1}_{\sphere^1}(\lp^{r-1},M)$, which 
is surjective by Step 1, and therefore so is $\iota^*$. 
Hence the last map in the sequence above is injective and the result follows.\\[2ex]
{\bf Step 3.} $H^j_{\sphere^1}(\Lambda^{r_2},M)\neq 0$ for all $j\ge \ind(K_{r_2})+2[(n-1)/2]+1$ with 
$j\equiv n \mod 2$. 

The critical manifolds $K_{r+1},\ldots, K_{r_2}$ do not contain $p$-iterated geodesics. 
By the index gap Lemma $H^j_{\sphere^1}(\Lambda^h,\Lambda^{h-1})=0$ for all $j\ge \ind(K_{r_2})+2[(n-1)/2]+1$ 
and $h=r+1,\ldots,r_2-1$. 
This readily implies $H^j(\Lambda^{r_2-1},M)\neq 0$ for all $j\ge \ind(K_{r_2})+2[(n-1)/2]+1$. 
The critical manifold $K_{r_2}\cong T^1M$ is regular and 
$H^{j}_{\sphere^1}(\Lambda^{r_2},\Lambda^{r_2-1})=0$ if $j\equiv n\mod 2$ while 
$H^{j}_{\sphere^1}(\Lambda^{r_2},\Lambda^{r_2-1})$ is torsion free if $j\equiv n+1\mod 2$. 
Clearly the result follows.
\\

Finally, the following step provides a contradiction to Corollary \ref{C:right-parity-R}. \\[2ex]
{\bf Step 4.} $H^{j_0}_{\sphere^1}(\Lambda M,M)\neq 0$ for $j_0= \ind(K_{r_2})+2[(n-1)/2]+1$. 

By the Index Gap Lemma all indices of critical manifold of energy $> e(K_{r_2})$ have indices 
$>j_0$. Furthermore the relative cohomology groups $H^{j_0+1}(\Lambda^h,\Lambda^{h-1})$ 
is torsion free for $h>r_2$ while $H^{j_0}_{\sphere^1}(\Lambda^{r_2},M)$ consists of nontrivial $p$-torsion.
Thus the map $H^{j_0}_{\sphere^1}(\Lambda M,M)\rightarrow 
H^{j_0}_{\sphere^1}(\Lambda^{r_2},M)$ is surjective.

\end{proof}

\begin{cor} Let $M$ be a Besse manifold. Then the energy function $E:\lp M\to \RR$ is rationally $\sphere^1$-equivariantly perfect, relatively to $M=\lp^0\In \lp M$ when all the negative bundles of all the critical sets of $E$ are orientable.
\end{cor}
\begin{proof}
We will prove this for $M$ even dimensional, the other case follows in the same way. It is enough to prove that for every $i$, the map
\[
H^{*-\ind(K_i)}_{\sphere^1}(K_i;\QQ)\simeq H^*_{\sphere^1}(\lp^i,\lp^{i-1};\QQ)\to H^*_{\sphere^1}(\lp^i;\QQ)
\]
is injective. We prove this, together with the statement that $H^{ev}_{\sphere^1}(\lp^i,M;\QQ)=0$, by induction on $i$.

For $i=0$ there is nothing to prove, so suppose that the both statements hold for $i-1$. By the long exact sequence of $(\lp^{i},\lp^{i-1},M)$ we have
\begin{align*}
H^{2m-\ind(K_i)}_{\sphere^1}(K_i;\QQ)&\to H^{2m}_{\sphere^1}(\lp^i,M;\QQ)\to H^{2m}_{\sphere^1}(\lp^{i-1},M;\QQ)\to \\
&\to H^{2m+1-\ind(K_i)}_{\sphere^1}(K_i;\QQ)\to H^{2m+1}_{\sphere^1}(\lp^i,M;\QQ)\to \ldots
\end{align*}
Since $M$ is even dimensional, $\ind(K)$ is odd, thus $H^{2m-\ind(K_i)}_{\sphere^1}(K_i;\QQ)=0$ by Proposition \ref{P:no-odd-degree}. Moreover, by the induction hypothesis $H^{2m}_{\sphere^1}(\lp^{i-1},M;\QQ)=0$, which gives
\begin{align*}
H^{2m}_{\sphere^1}(\lp^i,M;\QQ)&=0\\
0\to H^{2m+1-\ind(K_i)}_{\sphere^1}(K_i;\QQ)&\to H^{2m+1}_{\sphere^1}(\lp^i,M;\QQ),
\end{align*}
thus proving the induction step.
\end{proof}

%

\subsection{The general case}\label{SS:non-orientable-case}

We now remove the assumption that the negative bundles $\N\to K$ are all orientable. By Corollary \ref{C:orientability-spin}, this can only happen if the manifold $M$ has the integral cohomology of $\CC\PP^{2n}$. In particular, $M$ has even dimension, and therefore:
\\

\emph{For the rest of the section, we will assume that the manifold $M$ is even dimensional. In particular, by \cite{Wil}, $\ind(K)$ is odd for every critical set $K$ of the energy functional.}
\\

Let $K$ be a critical manifold with non orientable negative bundle $\N\to K$. In this case, we denote by $\delta:\hat{K}\to K$ the double cover such that $\N$ pulls back to an orientable bundle $\hat{\N}$ over $\hat{K}$. By Theorem \ref{T:orientability}, $\hat{K}$ can be realized as the quotient $\tilde{K}/H$ where $\hat{K}$ is the universal cover of $K$, and $H\In\pi_1(K)$ is the kernel of the homomorphism $A_*:\pi_1(K)\to \pi_1(\SO(n-1))\simeq \ZZ_2$ induced by the holonomy map $A:K\to \SO(n-1)$. Notice that, given an orbit $\gamma$ for the $\sphere^1$-action on $K$, $A(\gamma(t))$ is constant. Hence, $A_*([\gamma])=1$ and therefore the $\sphere^1$-action on $\N\to K$ lifts to $\hat{\N}\to \hat{K}$.
%
%

\begin{lem}\label{L:non-orientable}
Let $K, K'\In \lp M$ be critical sets of the energy functional, such that $K'=K^q$ for some $q$. Then
\begin{enumerate}
\item If $\N\to K$ is orientable, then $\N'\to K'$ is orientable as well.
\item If $\N\to K$ is non-orientable, then $\N'\to K'$ is orientable if and only if $q$ is even.
\item If both $\N\to K$ and $\N'\to K'$ are not orientable, the diffeomorphism $f:K\to K'$ sending $c$ to $c^q$ lifts to a $\ZZ_2$-equivariant diffeomorphism $\hat{K}\to \hat{K}'$.
\end{enumerate}
\end{lem}
\begin{proof}
The map $K\to K'$ sending $c\to c^q$ allows us to identify $K$ and $K'$. To prove 1) and 2), it is enough to observe that when $K$ has holonomy map $A$ and $K'=K^q$ then, under the identification $K\sim K'$, the holonomy map $A':K'\to \SO(n-1)$ equals $A^q$ and in particular $A'_*=q A_*$.

To prove 3), it is sufficient to further notice that when $\N\to K$ and $\N'\to K'$ are both non orientable, 
in particular $q$ is odd, and therefore the map $f_*:\pi_1(K)\to \pi_1(K')$ sends the kernel of $A_*$ 
isomorphically to the kernel of $A'_*$.

\end{proof}

Let $K$ be a critical manifold with non orientable negative bundle, and let $\hat{K}$ be the $2$-fold
cover defined above. 
The $\ZZ_2$-action on $\hat{K}$ induces a $\ZZ_2$-action on $H^*(\hat{K})$. Letting $g$ denote the generator
of $\ZZ_2$, 
we define
\begin{align*}
H^*(\hat{K})^{-\ZZ_2}&=\{x\in H^*(\hat{K})\mid g\cdot x=-x\}
\end{align*}

\begin{prop}\label{prop:Thom-non-orientable}
Let $K_i$ be a critical manifold for the energy functional, with non-orientable negative bundle.
Then if $R$ is a ring where $2$ is invertible, we have
\[
H^*(\lp^i,\lp^{i-1};R)\simeq H^{*-\ind(K_i)}(\hat{K}_i;R)^{-\ZZ_2}
\]
and the same holds for $\sphere^1$ equivariant cohomology.
\end{prop}
\begin{proof}
By excision, $H^*(\lp^i,\lp^{i-1};R)\simeq H^*(\N_i,\partial\N_i;R)$. Let $\hat{\eta}:\hat{\N}_i\to \hat{K}_i$ denote the lift of $\eta:\N_i\to K_i$. Then $(\hat{\N}_i,\partial\hat{\N}_i)\to (\N_i,\partial\N_i)$ is a $\ZZ_2$-cover as well and, since $2$ is invertible in $R$, by \cite[Thm. 2.4]{Bre} it induces an isomorphism
\[
H^*(\N_i,\partial\N_i;R)\simeq H^*(\hat{\N}_i,\partial\hat{\N}_i;R)^{\ZZ_2}.
\]
Moreover, since $\hat{\N}_i$ is orientable, by the Thom isomorphism there is a class $\hat{\tau}\in H^{\ind(K_i)}(\hat{\N}_i,\partial\hat{\N}_i;R)$ such that the map
\[
T:H^{k-\ind(K_i)}(\hat{K}_i;R)\to H^k(\hat{\N}_i,\partial\hat{\N}_i;R), \qquad \alpha\mapsto \hat{\eta}^*(\alpha)\cup \hat{\tau}
\]
induces an isomorphism of groups for every $q>0$. By construction, the Thom class $\hat{\tau}$ satisfies $g\cdot \hat{\tau}=-\hat{\tau}$, and thus
\begin{align*}
g\cdot T(\alpha)&=g\cdot(\hat{\eta}^*(\alpha)\cup \hat\tau)\\
&=(g\cdot \hat{\eta}^*(\alpha))\cup (g\cdot \hat\tau)\\
&=-\hat{\eta}^*(g\cdot \alpha)\cup \hat{\tau}=-T(g\cdot\alpha)
\end{align*}
Therefore, $T$ sends $H^{q-\ind(K_i)}(\hat{K}_i;R)^{-\ZZ_2}$ isomorphically into $H^q(\hat{\N}_i,\partial\hat{\N}_i;R)^{\ZZ_2}$. Therefore
\[
H^k(\N_i,\partial\N_i;R)\simeq H^k(\hat{\N}_i,\partial\hat{\N}_i;R)^{\ZZ_2}\simeq H^{q-\ind(K_i)}(\hat{K}_i;R)^{-\ZZ_2}.
\]
Because all the maps involved are $\sphere^1$-equivariant, and all the properties used hold for $\sphere^1$-equivariant cohomology as well, the result follows for $\sphere^1$-equivariant cohomology as well:
\[
H^k_{\sphere^1}(\N_i,\partial\N_i;R)\simeq H^{q-\ind(K_i)}_{\sphere^1}(\hat{K}_i;R)^{-\ZZ_2}.
\]
\end{proof}

We are now ready to modify the proof of Theorem \ref{T:perfectness} in the previous section, in the case of non orientable bundles. This time, since we do not have the Thom isomorphism at hand, we want to use the relative cohomology $H^*_{\sphere^1}(\lp^i,\lp^{i-1};\QQ)$ instead of $H^*_{\sphere^1}(K_i;\QQ)$, and prove by contradiction that it satisfies
\begin{equation}\label{E:contradiction}
H^{ev}_{\sphere^1}(\lp^i,\lp^{i-1};\QQ)=0.
\end{equation}
Supposing that this is not the case, then among the pairs which do not satisfy \eqref{E:contradiction}, we focus on the one whose corresponding critical set $C$ has minimal energy. Just as in the previous section, we provide the contradiction by showing show that for some prime $p$ big enough, the pair $(\lp^i,\lp^{i-1})$ corresponding to $C^p$ introduces some $p$-torsion element on $H^{ev}_{\sphere^1}(\lp M,M;R)$ which cannot be removed, contradicting Corollary \ref{C:right-parity-R} according to which $H^{ev}_{\sphere^1}(\lp M,M;R)=0$.

The following is the equivalent of Proposition \ref{P:H-odd-2}.

\begin{prop}\label{L:H-odd-non-orientable}
Let $e_i$ be a critical energy with critical manifold $K_i$, and let $p$ be a big enough prime. Then:
\begin{enumerate}
\item If $H^{ev}_{\sphere^1}(\lp^i,\lp^{i-1};\QQ)=0$ then for any critical energy $e_j$ with critical set $K_j=K_i^q$, $q$ odd, the group $H^{ev}(\lp^j,\lp^{j-1};R)$ contains no $p$-torsion.
\item If $H^{ev}_{\sphere^1}(\lp^i,\lp^{i-1};\QQ)\neq 0$, then for the critical energy $e_j$ with critical set $K_j=K_i^p$, the group $H^{2h}(\lp^j,\lp^{j-1};R)$ contains $p$-torsion for every $2h\geq \ind_0(K_i)$.
\end{enumerate}
\end{prop}
\begin{proof}
When $\N_i\to K_i$ is orientable, the result follows directly from Proposition \ref{P:H-odd-2} and the Thom isomorphism.

When $\N_i\to K_i$ is non-orientable, then we can repeat the same constructions in Proposition \ref{P:H-odd-2} to the $\ZZ_2$-coverings $\hat{K}_i$, $\hat{K}_i^q$ of $K_i$ and $K_i^p$, and obtain a $\ZZ_2$ equivariant $\sphere^1$-bundle
\begin{align}\label{E:bundle'}
\hat{\xi}:\; (\hat{K}^q_i)_{\sphere^1}\simeq \hat{K}_i\times_{\sphere^1}B\ZZ_q\to (\hat{K}_i)_{\sphere^1}\times B\sphere^1.
\end{align}
whose $\ZZ_2$-quotient is the bundle $\xi$ defined in Proposition \ref{P:H-odd-2}. The Gysin sequence of $\hat{\xi}$ is a $\ZZ_2$-equivariant long exact sequence
\begin{align*}
H^{2k+1}((\hat{K}_i)_{\sphere^1}\times B\sphere^1;R)\stackrel{\hat{\xi}^*}{\lra} H^{2k+1}_{\sphere^1}(\hat{K}^q;R)\stackrel{\hat{\xi}_!}{\lra} H^{2k}((\hat{K}_i)_{\sphere^1}\times B\sphere^1;R)
\end{align*}
equivalent to \eqref{E:Gysin}. Arguing in the same way as in Proposition \ref{P:H-odd-2} and taking the $-\ZZ_2$-invariant part, we can see that $H^{odd}_{\sphere^1}(\hat{K}^q_i;R)^{-\ZZ_2}$ (which equals $H^{ev}_{\sphere^1}(\lp^j,\lp^{j-1};R)$) cannot have $p$-torsion unless  $H^{odd}_{\sphere^1}(\hat{K};\QQ)^{-\ZZ_2}$ ($=H^{ev}_{\sphere^1}(\lp^i,\lp^{i-1};\QQ)$) is nonzero.

On the other hand, if $H^{odd}_{\sphere^1}(\hat{K}_i;\QQ)^{-\ZZ_2}\neq 0$ then we can find some torsion-free element $x\in H^{odd}_{\sphere^1}(\hat{K}_i;R)^{-\ZZ_2}$ not divisible by $p$ and such that $x\cup \hat{e}=0$, where $\hat{e}\in H^{2}_{\sphere^1}(\hat{K}_i;R)$ is the Euler class of $\hat{K}_i\to \hat{K}_{\sphere^1}$. Then, again as in Proposition \ref{P:H-odd-2}, one can prove that for every $k>0$, the element $\hat{\xi}^*(x\otimes c^k)\in H^{odd}_{\sphere^1}(\hat{K}_i^p;R)^{-\ZZ_2}$ is a non trivial $p$-torsion element.

\end{proof}

The new version of Proposition \ref{P:no-odd-degree} is the following: 
\begin{prop}\label{P:no-ev-degree}
For every critical energy $e_i$, one has $H^{ev}_{\sphere^1}(\lp^i,\lp^{i-1};\QQ)=0$.
\end{prop}
The proof is, for a large part, the same as the one of Proposition \ref{P:no-odd-degree}: we will give a sketch of the proof of Proposition \ref{P:no-ev-degree}, by focusing on the parts that differ from Proposition  \ref{P:no-odd-degree}.
\begin{proof}
As in Proposition \ref{P:no-odd-degree}, we choose a large prime and consider the localization $S=\ZZ[{1\over q}\mid q\textrm{ prime}\neq p]$ as coefficient ring.
We argue by contradiction and consider the smallest length $l$ in correspondence to which one has $H^{ev}_{\sphere^1}(\lp^i,\lp^{i-1};\QQ)\neq 0$.
From Proposition \ref{prop:Thom-non-orientable} and Proposition \ref{L:H-odd-non-orientable}, one can see that the critical manifold $K_i$ has one of the following forms:
\begin{itemize}
\item $K_i=C_j$ for some $C\in \{C_1,\ldots C_m\}$ with $\N\to C$ orientable and $H^{odd}_{\sphere^1}(C;\QQ)\neq0$.
\item $K_i=C_j$ for some $C\in \{C_1,\ldots C_m\}$ with $\N\to C$ non orientable and $H^{odd}_{\sphere^1}(\hat{C};\QQ)^{-\ZZ_2}\neq0$.
\item $K_i=C_j^2$ for some $C\in \{C_1,\ldots C_m\}$ with $\N\to C$ non orientable, $H^{odd}_{\sphere^1}(\hat{C};\QQ)^{-\ZZ_2}=0$ and $H^{odd}_{\sphere^1}(C;\QQ)\neq0$.
\end{itemize}

Consider the set $\mathcal{S}$ of critical manifolds $K_j$ of the form $K_j=C$, $C\in \{C_1,\ldots, C_m\}$, or $K_j=C^2$ with $\N\to C$ non orientable, and pick an $\epsilon$ small enough that there are no critical sets in $\mathcal{S}$ of length $l'\in[l-\epsilon,l+\epsilon]\setminus \{l\}$. By choosing $p$ large enough, we can assume that there are integers $u_1<u_2$ such that $p(l-\epsilon)<u_1L<pl<u_2L<p(l+\epsilon)$.

Let now $K_r$ denote the critical set $(K_i)^p$, and $K_{r_1},K_{r_2}$ the critical manifolds of length $u_1L$ and $u_2L$, respectively. By Proposition \ref{L:H-odd-non-orientable}, it follows that $K_h$ does not contain $p$-iterates for $h=r_1,\ldots,r-1$ and therefore $H^{2j}_{\sphere^1}(\lp^h,\lp^{h-1})=0$ for every $2j\geq \ind(K_h)+(n-1)$. Using this, the steps of Proposition  \ref{P:no-odd-degree} follow identically.
\end{proof}


\section{The proof of the Main Theorem}
\label{S:proof-main-thm}

Let $\sphere^n$, $n>3$, be a topological sphere enowed with a Besse metric, 
and let $L$ be the common period of the geodesics. We want to prove by contradiction that all geodesics have the same length $L$ or, equivalently, that the geodesic flow $\sphere^1\curvearrowright T^1\sphere^n$ acts freely. If not, there are critical sets of the energy functional $E:\lp \sphere^n\to \RR$ which consist of geodesics of length $L/m$, $m\in \ZZ$, and they can be identified with the fixed point set, in $T^1\sphere^n$, of the subgroup $\ZZ_m\In \sphere^1$. Since these sets have even codimension in $T^1\sphere^n$, in particular every critical set of positive energy has odd dimension.

By Proposition \ref{P:index-gap}, the critical set $C$ of lowest index consists of geodesics of length $L/m$ for some integer $m$ and, by the discussion above, it must be
\[
\dim C\leq 2n-3.
\]

We now divide the discussion into two cases, according to whether $n$ is even or odd.
\\

If $n$ is even, the integral cohomology and rational $\sphere^1$-equivariant cohomology groups of $(\lp \sphere^n,\sphere^n)$ are the following
\begin{align*}
H^q(\lp \sphere^n,\sphere^n;\ZZ)&=
\left\{\begin{array}{ll}\ZZ & q=(2k-1)(n-1),\; k\geq 1\\ \ZZ & q=(2k+1)(n-1)+1,\; k\geq 1\\
\ZZ_2 & q=2k(n-1)+1,\; k\geq 1\\ 0 &\textrm{otherwise}\end{array}\right.\\
H^q_{\sphere^1}(\lp \sphere^n,\sphere^n;\QQ)&=
\left\{\begin{array}{ll}\QQ & q\geq (n-1)\textrm{ odd},\; q\neq 3(n-1),\,5(n-1),\ldots \\
\QQ^2 & q=3(n-1),\,5(n-1),\ldots\\ 0 &\textrm{otherwise}\end{array}\right. 
\end{align*}

By the perfectness of the energy functional, $C$ consists of one component only, and $\ind(C)=n-1$. Moreover, the $\sphere^1/\ZZ_m$-action on $C$ is free, otherwise 
there would be some closed geodesics $c\notin C$ such that $c^k\in C$ for some $k$, but the Bott iteration formula in Section \ref{S:index-gap} would give $\ind(c)\leq \ind(c^k)=\ind(C)$ contradicting the minimality of $\ind(C)$. In particular, $\OO(2)/\ZZ_m\simeq \OO(2)$ acts freely as well on $C$.

The quotient (manifold) $C/\sphere^1$ is embedded in $T^1\sphere^n/\sphere^1$, which is a symplectic orbifold  (cf. \cite{We}). It is easy to see that the symplectic form on $T^1\sphere^n/\sphere^1$ restricts to a symplectic form on  $C/\sphere^1$, and therefore $C/\sphere^1$ is a symplectic manifold. In particular,
\begin{equation}\label{E:ineq1}
\dim H^{2q}_{\sphere^1}(C;\QQ)= \dim H^{2q}(C/\sphere^1;\QQ)\geq 1
\end{equation}
for every $2q\leq \dim (C/\sphere^1)$. However, by the perfectness of the energy functional in rational equivariant cohomology, for any $q\leq 2n-3$ we have
\begin{align}\label{E:ineq2}
\dim H^{q}_{\sphere^1}(C;\QQ)&=\dim H_{\sphere^1}^{q+(n-1)}(\lp^{e+\epsilon},\lp^{e-\epsilon};\QQ)\\
							&\leq \dim H^{q+(n-1)}_{\sphere^1}(\lp \sphere^n,\sphere^n;\QQ)=\left\{\begin{array}{ll} 1& q\textrm{ even}\\ 0&q\textrm{ odd}\end{array}\right. \nonumber
\end{align}

From inequalities \eqref{E:ineq1} and \eqref{E:ineq2} it follows that for any $q\leq \dim(C/\sphere^1)=\dim(C) -1$, 
\[
H^{q}_{\sphere^1}(C;\QQ)=H^{q+(n-1)}_{\sphere^1}(\lp^{e+\epsilon},\lp^{e-\epsilon};\QQ)=H^{q+(n-1)}_{\sphere^1}(\lp \sphere^n,\sphere^n;\QQ)
\]
where $e=E(C)$.
Again by perfectness of the energy functional, it follows that every critical set different from $C$ cannot contribute to the rational equivariant cohomology in degrees $\leq (n-1)+ \dim(C)-1$, and in particular the index of every critical set different from $C$ must be $\geq (n-1)+\dim(C)$.

The index of the critical sets, however, does not depend on the cohomology we are using. In particular, if we now switch to regular integral cohomology, we still have that the only contribution to the cohomology $H^q(\lp\sphere^n,\sphere^n;\ZZ)$ in degrees $q\leq (n-1)+\dim(C)-1$ is given by $H^{q-(n-1)}(C;\ZZ)$ and therefore
\[
H^{q}(C;\ZZ)=H^{q+(n-1)}(\lp \sphere^n,\sphere^n;\ZZ)\qquad \forall q=0,\ldots \dim(C)-1.
\]
In particular, $H^{q}(C;\ZZ)=0$ for every $q=1,\ldots m_0=\min\{\dim(C)-1, n-1\}$. For $n\geq 4$, we have $m_0\geq {1\over 2}\dim(C)+1$ and therefore, by Poincar\'e duality, $C$ is an integral cohomology sphere. However, $\ZZ_2\times\ZZ_2\In \OO(2)$ acts freely on $C$ and this contradicts the well-known result in Smith's theory, that a finite abelian group acting freely on an integral cohomology sphere must be cyclic.
\\

If $n$ is odd, the integral cohomology and the rational $\sphere^1$-equivariant cohomology of $(\lp \sphere^n,\sphere^n)$ are as follows:
\begin{align*}
H^q(\lp \sphere^n,\sphere^n;\ZZ)&=\left\{\begin{array}{ll}\ZZ & q=k(n-1)\textrm{ or } q=(k+1)(n-1)+1,\; k\geq 1\\ 0 &\textrm{otherwise}\end{array}\right.\\
H^q_{\sphere^1}(\lp \sphere^n;\QQ)&=\left\{\begin{array}{ll}\QQ & q\geq (n-1)\textrm{ even},\; q\neq 2(n-1),\,3(n-1),\ldots \\ \QQ^2 & q=2(n-1),\,3(n-1),\ldots\\ 0 &\textrm{otherwise}\end{array}\right. 
\end{align*}

As in the previous case $\ind(C)=n-1$, $C$ is the unique critical set with minimal index, and $\OO(2)/\ZZ_m\simeq \OO(2)$ acts freely on $C$. Moreover, $C/\sphere^1$ is a symplectic manifold and
\begin{equation}\label{E:ineq3}
\dim H^{2q}(C/\sphere^1;\QQ)=\dim H^{2q}_{\sphere^1}(C;\QQ)\geq 1\qquad \forall q\leq {1\over 2}\dim(C/\sphere^1).
\end{equation}

By the perfectness of the energy functional in rational equivariant cohomology, for any $q\leq n-2$ we have

\begin{align}\label{E:ineq4}
\dim H^{q}_{\sphere^1}(C;\QQ)&=\dim H_{\sphere^1}^{q+(n-1)}(\lp^{e+\epsilon},\lp^{e-\epsilon};\QQ)\\
							&\leq \dim H^{q+(n-1)}_{\sphere^1}(\lp \sphere^n,\sphere^n;\QQ)=\left\{\begin{array}{ll} 1& q\textrm{ even}\\ 0&q\textrm{ odd}\end{array}\right.\nonumber
\end{align}

For any $q\leq m_0=\min\{\dim(C/\sphere^1), n-2\}$ we then have 
\begin{equation}\label{E:filling}
H^{q}_{\sphere^1}(C;\QQ)=H^{q+(n-1)}_{\sphere^1}(\lp^{e+\epsilon},\lp^{e-\epsilon};\QQ)=H^{q+(n-1)}_{\sphere^1}(\lp \sphere^n,\sphere^n;\QQ)
\end{equation}

For $n\geq 3$ we have $m_0\geq {1\over 2}\dim(C/\sphere^1)$ and, by Poincar\'e duality, we obtain that $H^*(C/{\sphere^1};\QQ)=H^*(\CC\PP^{\dim(C/\sphere^1)/2};\QQ)$.

We claim that $\dim(C/\sphere^1)\leq n-1$. If this was not the case, the critical set $C'$ with second lowest index, would have index $2(n-1)$. Call $e'$ its energy. If $\dim C'=1$, then $C'$ consists of at least two geodesics, because for any geodesic $\gamma(t)$ in $C'$, the inverted geodesic $\gamma(-t)$ belongs to $C'$ as well. Therefore $H^0_{\sphere^1}(C';\QQ)=H^{2(n-1)}_{\sphere^1}(\lp^{e'+\epsilon},\lp^{e'-\epsilon};\QQ)$ would have dimension $\geq 2$, hence by perfectness
\begin{align*}
2&=\dim H^{2(n-1)}_{\sphere^1}(\lp \sphere^n,\sphere^n;\QQ)\\
&\geq H^{2(n-1)}_{\sphere^1}(\lp^{e+\epsilon},\lp^{e-\epsilon};\QQ)+ H^{2(n-1)}_{\sphere^1}(\lp^{e'+\epsilon},\lp^{e'-\epsilon};\QQ)\geq 3
\end{align*}
which gives a contradiction. If $\dim C'/\sphere^1\geq 1$, the quotient $C'/\sphere^1$ would be a symplectic orbifold just as $C$ and hence it would satisfy
\[
\dim H^2(C'/\sphere^1;\QQ)=\dim H^{2n}_{\sphere^1}(\lp^{e'+\epsilon},\lp^{e'-\epsilon};\QQ)\geq 1,
\]
but again by perfectness we would then have
\begin{align*}
1&=\dim H^{2n}_{\sphere^1}(\lp \sphere^n,\sphere^n;\QQ)\\
&\geq \dim H^{2n}_{\sphere^1}(\lp^{e+\epsilon},\lp^{e-\epsilon};\QQ)+\dim H^{2n}_{\sphere^1}(\lp^{e'+\epsilon},\lp^{e'-\epsilon};\QQ)\geq 2
\end{align*}
which would provide a contradiction as well.

Therefore $\dim(C/\sphere^1)\leq n-1$ and, by \eqref{E:filling}, every other critical set has index $\geq (n-1)+ \dim(C)-1$. Then shifting our attention to integral cohomology, the only contribution to $H^q(\lp \sphere^n,\sphere^n;\ZZ)$ in degrees $q\leq (n-1)+\dim(C)-1$ is given by $H^{q-(n-1)}(C;\ZZ)$. In particular,
\[
H^q(C;\ZZ)=H^{q+(n-1)}(\lp \sphere^n,\sphere^n;\ZZ)=\left\{\begin{array}{ll} \ZZ &q=0\\ 0& q=1,\ldots \dim(C)-1\end{array}\right.
\]
For $n\geq 3$, this covers more than half of the cohomology of $C$ and therefore $C$ is an integral cohomology sphere. As in the previous case, the fact that $\ZZ_2\oplus \ZZ_2\In \OO(2)/\ZZ_m$ acts freely on $C$ provides a contradiction.

%
%

\bigskip

\medskip

\appendix


\section{Small subsets of $\Sp(n-1,\omega)$}\label{A:small subsets}
Let $(\RR^{2(n-1)},\omega)$ be the symplectic vector space, and let $\Sp(n-1,\omega)=\{P\in \GL(2(n-1),\RR)\mid \omega(Ax,Ay)=\omega(x,y)\}$ be the real symplectic group. 
In this appendix we focus on the subspaces of symplectic matrices with real eigenvalues of higher geometric multiplicity.
%

\begin{lem}\label{L:Decomposition}
Let $P\in \Sp(n-1,\omega)$ and let $\lambda$ be a positive real eigenvalue of $P$ of algebraic multiplicity $a$.
\begin{itemize}
\item[a)] If $\lambda\neq 1$ then up to conjugation with an element of $\Sp(n-1,\omega)$, the matrix $P$ can be written as $P=\diag(\lambda U^{\textrm{tr}}, (\lambda U)^{-1}, R)$ where $U\in \GL(a,\RR)$ is unipotent and $R\in \Sp(n-a-1,\omega)$.
\item[b)] If $\lambda=1$ then up to conjugation with some element of $\Sp(n-1,\omega)$, the matrix $P$  can be written as $P=\diag(U, R)$ where $U\in \Sp(a,\omega)$ is unipotent and $R\in \Sp(n-a-1,\omega)$.
\end{itemize}
\end{lem}
\begin{proof}
By the so-called refined Jordan decomposition, there are commuting matrices $P_s, P_u\in \Sp(n-1,\omega)$ such that $P_s$ is diagonalizable over $\CC$, $P_u$ is unipotent, and $P=P_sP_u$. In particular, $P_s$ has the same eigenvalues of $P$ with the same algebraic multiplicities.

Let $E_1$ denote the direct sum of the eigenspaces of $P_s$ of eigenvalues $\lambda$ and $\lambda^{-1}$, and $E_2$ the sum of the other eigenspaces. Since $E_1$ and $E_2$ are symplectic subspaces (see for example \cite{BTZ}) then, up to conjugation with a symplectic matrix, we can write $P_s=\diag(\lambda \Id_a,\lambda^{-1}\Id_a, R_1)$ (if $\lambda\neq 1$) or $P_s=\diag(\Id_{2a},R_1)$ (if $\lambda=1$) for some $R_1\in \Sp(n-a-1,\omega)$.

Since $P_u$ commutes with $P_s$, it can be written either as $P_u=\diag(U^{\textrm{tr}},U^{-1}, R_2)$ for some $U\in \GL(a,\RR)$ unipotent and $R_2\in \Sp(n-a-1,\omega)$ (if $\lambda\neq 1$), or $P_u=\diag(U, R_2)$ for some $U\in \Sp(a,\omega)$ unipotent and $R_2\in \Sp(n-a-1,\omega)$ (if $\lambda=1$).

Since $P=P_uP_s$ we have proved the lemma.
\end{proof}

Given an algebraic group $G\In \GL(N,\RR)$, recall that a torus $T\In G$ is a connected, abelian subgroup whose elements are diagonalizable over $\CC$. Every algebraic group admits at least one torus of maximal dimension, called \emph{maximal torus}, which is unique up to conjugation by an element of $G$, and the \emph{rank} of $G$, denoted $\textrm{rk}\, G$, is defined as the dimension of a maximal torus of $G$. We will be mostly concerned with $G=\Sp(N,\omega)$, in which case $\textrm{rk}\, G=N$. The following Lemma is a consequence of well-known results, but we could not find any reference in the literature.

\begin{lem}\label{L:unip-codim}
Given an algebraic group $G\In \GL(N,\RR)$, the set $G_u$ of unipotent elements in $G$ is an affine variety of codimension equal to the rank of $G$.
\end{lem}
\begin{proof}
The set $G_u$ is invariant under the action of $G$ by conjugation. Fixing  $B$ a Borel (i.e., maximal connected solvable) subgroup of $G$, let $B_u$ denote the subset of unipotent elements in $B$. Every $G$-orbit meets $B$ at least once by \cite[11.10]{Borel}, and therefore the map $G\times B_u\to G_u$ sending $(g,A)$ to $gAg^{-1}$ is surjective. The normalizer $N(B_u)=\{g\in G\mid gB_ug^{-1}\In B_u\}$ coincides with $B$ by a theorem of Chevalley \cite[11.16]{Borel} and therefore $\dim B_u+\dim G-\dim B=\dim G_u$. By \cite[10.6]{Borel} $B_u$ is normal in $B$ and $B/B_u$ is isomorphic to a maximal torus, in particular $\dim B=\dim B_u+\textrm{rk}\,G$. With the equation before, we obtain $\dim G_u=\dim G-\textrm{rk}G$.
\end{proof}

Let $\Sp_1(n-1,\omega)$ denote the space of symplectic matrices whose positive real eigenvalues have geometric multiplicity $1$. The next result shows that the complement of $\Sp_1(n-1,\omega)$ in $\Sp(n-1,\omega)$ has codimension at least 3.

\begin{prop}\label{P:geom-mult}
The set of matrices $P\in \Sp(n-1,\omega)$ with some eigenvalue $\lambda\in (0,1]$ of geometric multiplicity $>1$ has codimension $\geq 3$ in $\Sp(n-1,\omega)$.
\end{prop}

\begin{proof}
Given $\lambda$ in $(0,1]$ let $\mathcal{M}_{\lambda}$ denote the space of matrices in $\Sp(n-1,\omega)$ with eigenvalue $\lambda$ of geometric multiplicity at least 2. This set can be also described as
\[
\mathcal{M}_{\lambda}=\{X\in \Sp(n-1,\omega)\mid \textrm{rk}(X-\lambda I)\leq 2(n-2)\}
\]
from which it follows that $\mathcal{M}_\lambda$ is an algebraic variety, and we can talk about its dimension. To prove the lemma, it is enough to prove that the codimension of $\mathcal{M}_{\lambda}$ is $\geq 4$ for $\lambda\neq 1$, and $\geq 3$ if $\lambda=1$.

We also define $\mathcal{M}_\lambda(n_1,n_2)$ to be the subspace of matrices $P$ in $\mathcal{M}_\lambda$ such that the generalised eigenspace with eigenvalue $\lambda$ can be written as a sum of two $P$-invariant subspaces  of dimension $n_1,n_2$. The set $\mathcal{M}_\lambda$ consists of a finite union of the $\mathcal{M}_\lambda(n_1,n_2)$ and it suffices to show that each of them has the required codimension.

Suppose first that $\lambda\neq 1$. Fixing one $\mathcal{M}=\mathcal{M}_\lambda(n_1,n_2)$, let $\Sigma\In \mathcal{M}_{\lambda}$ denote the subset of matrices $P$ that, in some fixed basis, can be written as $P=\diag(P_1,R)$, with $P_1, R$ both symplectic, and $P_1$ decomposing further as $\diag(\lambda U_1^{\textrm{tr}},  (\lambda U_1)^{-1},\lambda U_2^{\textrm{tr}}, (\lambda U_2)^{-1})$ where $U_1\in \GL(n_1,\RR)$, $U_2\in \GL(n_2,\RR)$ have the form
\[
\left(\begin{array}{cccc}1 &  &  &  \\ 1& 1 &  &  \\ & \ddots & \ddots &  \\ &  & 1 & 1\end{array}\right)
\]

The set $\mathcal{M}$ is preserved under the action of $\Sp(n-1,\omega)$ on itself by conjugation and, by Lemma \ref{L:Decomposition}, every orbit meets $\Sigma$ in at least one point. Therefore, the map $\Sigma\times \Sp(n-1,\omega)\to \mathcal{M}$ is surjective and, letting $N(\Sigma)=\{A\in \Sp(n-1,\omega)\mid A\Sigma A^{-1}\In \Sigma\}$ denote the normalizer of $\Sigma$, we have
\begin{equation}\label{E:formula-M}
\dim \mathcal{M}=\dim \Sp(n-1,\omega)+ \dim \Sigma -\dim N(\Sigma).
\end{equation}
Clearly a matrix $P=\diag(P_1,R)$ in $\Sigma$ is uniquely determined by $R\in \Sp(n'-1,\omega)$, $n'=n-n_1-n_2$, and therefore $\dim \Sigma=\dim\Sp(n'-1,\omega)$.

We now compute $N(\Sigma)$. Suppose that $n_1\leq n_2$, and let $\mathcal{A}\In \GL(n_1+n_2,\RR)$ be the set of matrices such that
\[
A^{\textrm{tr}}=\left(\begin{array}{ccc|cccc}
a_1 	& \cdots	& a_{n_1} & 0		& b_1 	& \cdots 	& b_{n_1} \\ 
	& \ddots	& \vdots 	& \vdots 	&  		& \ddots 	& \vdots \\ 
	&  & a_1	& 0 		& 		&  		& b_1 \\ \hline
c_1 	& \cdots 	& c_{n_1} & d_1	& d_2 	& \cdots 	&	d_{n_2} \\
	& \ddots 	& \vdots 	&  		& d_1 	& \ddots 	& \vdots \\
	&  		& c_1	&  		& 		& \ddots 	& d_2 \\
0 	& \cdots 	& 0 		&  		&  		&  		& d_1\end{array}\right)
\]
Clearly $\dim \mathcal{A}=3n_1+n_2\geq 4$. For any matrix $A\in \mathcal{A}$ and $B\in \Sp(n'-1,\omega)$, the matrix $\diag(A^{\textrm{tr}},A^{-1},B)$ lies in $N(\Sigma)$.  In particular, $\dim N(\Sigma)\geq \dim \mathcal{A}+\Sp(n'-1,\omega)$ and therefore $\dim \mathcal{M}\leq \dim \Sp(n-1,\omega)-\dim\mathcal{A}\leq \Sp(n-1,\omega)-4$.

If $\lambda=1$ then any $P\in \mathcal{M}_{\lambda}$ can be written, in a suitable symplectic basis, as $\diag(U,R)$ where $U\in \Sp(n_0,\omega)$, for some $n_0$, is unipotent with at least two linearly independent eigenvectors, and $R\in \Sp(n-n_0-1,\omega)$. We now define $\Sigma$ to be the set of matrices that, under the same fixed basis, can be written as $\diag(U',R')$ for some $R'\in \Sp(n'-1,\omega)$ and some unipotent matrix $U'\in \Sp(n_0,\omega)$. If we let $\Sp(n_0,\omega)_u$ denote the set of unipotent matrices in $\Sp(n_0,\omega)$, we have
\[
\dim \Sigma= \dim \Sp(n_0,\omega)_u+\dim \Sp(n-n_0-1,\omega).
\]
where $\Sp(n_0, \omega)_u$ denote the unipotent matrices in $\Sp(n_0, \omega)$. The normalizer $N(\Sigma)$ contains the matrices of the form $\diag(P_1,R')$ with $P_1\in \Sp(n_0,\omega)$ and $R'\in \Sp(n-n_0-1,\omega)$, thus
\[
\dim N(\Sigma)\geq \dim \Sp(n_0,\omega)+\dim \Sp(n-n_0-1,\omega).
\]

Once again $\Sp(n-1,\omega)$ acts on $\mathcal{M}_\lambda$ and by Lemma \ref{L:Decomposition} every orbit meets $\Sigma$. In particular, $\mathcal{M}_\lambda$ is contained in the space spanned by the orbits of $\Sigma$, and $\dim \mathcal{M}_\lambda\leq\dim \Sp(n-1,\omega)-(\dim N(\Sigma)-\dim \Sigma)$. By the computation above and Lemma \ref{L:unip-codim}, we have
\begin{align*}
\dim \mathcal{M}_\lambda&\leq \dim \Sp(n-1,\omega)- \mathrm{rk}\,\Sp(n_0,\omega)\\ &= \dim\Sp(n-1,\omega)-n_0.
\end{align*}
The codimension of $\mathcal{M}_\lambda$ is then $\geq 3$, unless possibly when $n_0=1$ or $2$.

If $n_0=1$, then every matrix in $\mathcal{M}_{\lambda}$ can be written, under some basis, as
\begin{equation}\label{E:n=1}
\diag(\Id_2,R),\qquad R\in \Sp(n-2,\omega)
\end{equation}
Fixing a basis and letting $\Sigma$ denote the space of matrices that, in the fixed basis, can be written as in \eqref{E:n=1}, we have that  $\dim\Sigma=\dim\Sp(n-2,\omega)$ and the normalizer $N(\Sigma)$ contains $\Sp(1,\omega)\times\Sp(n-2,\omega)$. Therefore, $\dim N(\Sigma)-\dim \Sigma\geq 3$. Using Equation \eqref{E:formula-M}, we obtain $\dim \mathcal{M}_\lambda\leq \dim\Sp(n-1,\omega)-3$.

If $n_0=2$, then every matrix in $\mathcal{M}_{\lambda}$ can be written, under some basis, as
\begin{equation}\label{E:n=2}
U= \diag(U_1,U_2)\qquad U_i=\left(\begin{array}{cc}1 &  \\ \sigma_i & 1\end{array}\right),\, \sigma_i\in \{-1,0,1\}.
\end{equation}
Fixing a basis and letting $\Sigma$ denote the space of matrices that, in the fixed basis, can be written as in \eqref{E:n=2}, we have $\dim\Sigma=\dim \Sp(n-3,\omega)$. If for example $\sigma_1=\sigma_2=1$ then $N(\Sigma)$ contains all the matrices of the form $\diag(P_1,R)$ where $R\in \Sp(n-3,\omega)$ is any matrix, and $P_1\in \Sp(2,\omega)$ has the form

\[
P_2=\left(\begin{array}{cc|cc}
\cos\theta 	& a	& -\sin\theta 	& b \\ 
	& \cos\theta	& 		& -\sin\theta \\ \hline
\sin\theta 	& c 	& \cos\theta	& d  \\
	&  \sin\theta	&  		& \cos\theta \end{array}\right),
\]
where $a,b,c,d$ satisfy the linear equation $(a+b)\cos\theta=(c-d)\sin\theta$. 
Therefore $\dim N(\Sigma)\geq \dim\Sp(n-3,\omega)+4$ and, by Equation \eqref{E:formula-M}, 
we obtain $\dim\mathcal{M}_\lambda\leq \dim\Sp(n-1,\omega)-4$. The same computations can be checked for 
the other values of $\sigma_1$ and $\sigma_2$.
\end{proof}

Let $\mathcal{G},\mathcal{G}_1$ denote the subspaces of $\Sp(n-1,\omega)\times\RR_+$ given by
\begin{align*}
\mathcal{G}&=\left\{(P,\lambda)\in \Sp(n-1,\omega)\times\RR_+\mid \dim \ker(P-\lambda\, \Id)\leq 1 \right\}\\
\mathcal{G}_1&=\left\{(P,\lambda)\in \Sp(n-1,\omega)\times\RR_+\mid \dim \ker(P-\lambda\, \Id)= 1 \right\}
\end{align*}
Clearly $\mathcal{G}$ is open and dense in $\Sp(n-1,\omega)\times\RR_+$, $\mathcal{G}_1\In \mathcal{G}_0$ and, by construction, we have $\Sp_1(n-1,\omega)\times \RR_+\In \mathcal{G}$.
\begin{prop}\label{P:map-surj}
The map
\begin{align*}
\chi:\mathcal{G}&\lra \RR\\
(P,\mu)&\lmt \det(P-\mu \,\Id)
\end{align*}
is a submersion in a neighbourhood of $\mathcal{G}_1$.
\end{prop}
\begin{proof}
We are going to prove a stronger statement. In fact, we prove that for any $(S,\lambda)\in \mathcal{G}_0$ 
we can find a vector $v_{\scriptscriptstyle(S,\lambda)}\in T_S\Sp(n-1,\omega)$ such that
$d_{\scriptscriptstyle(S,\lambda)}\chi(v_{\scriptscriptstyle(S,\lambda)})> 0$.

Let $a$ denote the algebraic multiplicity of $\lambda$ in $S$. By Lemma \ref{L:Decomposition}, there is a
symplectic basis such that $S$ can be written as $S=\diag(S_1,S_2)$ where $S_1\in \Sp(a,\omega)$ only contains the eigenvalues $\lambda, \lambda^{-1}$ and $S_2\in \Sp(n-1-a,\omega)$ has eigenvalues different from $\lambda$ and $\lambda^{-1}$.

If $\lambda\neq1$ then by Lemma \ref{L:Decomposition} we can write 
$S_1=\diag(\lambda U^{\textrm{tr}}, (\lambda U)^{-1})$, where
\[
\lambda U^{\textrm{tr}}=\left(\begin{array}{cccc}\lambda & 1 &   &   \\  & \lambda & \ddots &   \\  &   & 
\ddots & 1 \\  &   &   & \lambda\end{array}\right).
\]
If $\lambda>1$ let $v=\diag(-E_{a,1},E_{1,a})\in \mathfrak{sp}(a,\omega)$ otherwise let
$v=(-1)^{a}\diag(-E_{a,1},E_{1,a})$. In either case, let $v'=\diag(v,0)\in \sp(n-1,\omega)$ and 
$v_{\scriptscriptstyle(S,\lambda)}={L_S}_*v'\in T_S\Sp(n-1,\omega)$, one can compute
\[
d_{\scriptscriptstyle(S,\lambda)}\chi(v_{\scriptscriptstyle(S,\lambda)})=\lambda|\lambda-\lambda^{-1}|^a> 0.
\]
If $\lambda=1$, then $S_1$ can be written in the following block form
\[
S_1=\left(\begin{array}{c  c} U^{-1} &    \\  U^{\textrm{tr}}T & U^{\textrm{tr}}\end{array}\right)
\]
where $T$ is a symmetric matrix, and
\[
U^{\textrm{tr}}=\left(\begin{array}{ccc}1 &  \cdots & 1 \\ &  \ddots & \vdots  \\
&  & 1\end{array}\right),\qquad U^{-1}=\left(\begin{array}{ccc}1  &  &  \\-1 & \ddots &    \\
& -1 & 1\end{array}\right)
\]
In order for $S_1$ to have geometric multiplicity 1, it must be 
$(UT)_{a,a}=c\neq 0$. Without loss of generality we can suppose that the sign of $c$ is $(-1)^{a-1}$. Define
\[
v=\left(\begin{array}{c | c} 0 & E_{1,1}   \\ \hline 0 & 0 \end{array}\right)\in \mathfrak{sp}(a,\omega).
\]
Letting $v'=\diag(v,0)\in \sp(n-1,\omega)$ and $v_{\scriptscriptstyle(S,\lambda)}={L_S}_*v'\in T_S\Sp(n-1,\omega)$
we can easily compute that
\[
d_{\scriptscriptstyle(S,\lambda)}\chi(v_{\scriptscriptstyle(S,\lambda)})= (-1)^{a-1}c> 0.
\]
\end{proof}

From Proposition \ref{P:map-surj} we obtain the following stronger, more global result.
\begin{prop}\label{P:vector-field}
There exists a vector field $V$ on $\Sp_1(n-1,\omega)$ such that for every $S\in \Sp_1(n-1,\omega)$ and 
every real eigenvalue $\lambda$ of $S$, $d_{\scriptscriptstyle(S,\lambda)}\chi(V)>0$.
\end{prop}
\begin{proof}
Given $S\in \Sp_1(n-1,\omega)$ and $\lambda\in (0,1]$ a real eigenvalue of $S$, let $v_{\scriptscriptstyle(S,\lambda)}\in T_S\Sp_1(n-1,\omega)$ be the vector constructed in the previous proposition, so that $d_{\scriptscriptstyle(S,\lambda)}\chi(v_{\scriptscriptstyle(S,\lambda)})>0$. It can be easily checked that, at the point $(S,\lambda^{-1})$,  the differential of $\chi$ is
\[
d_{\scriptscriptstyle(S,\lambda^{-1})}\chi(v_{\scriptscriptstyle(S,\lambda)})=-\lambda^{1-2a}(\lambda-\lambda^{-1})^a> 0
\]
and moreover for any other eigenvalue $\mu$ of $S$ different from $\lambda$ and $\lambda^{-1}$, one has $d_{\scriptscriptstyle(S,\mu)}\chi(v_{\scriptscriptstyle(S,\lambda)})=0$. In particular, letting $v_S=\sum_{\lambda} v_{\scriptscriptstyle(S,\lambda)}$, where the sum is taken over all the real eigenvalues of $S$ in $(0,1]$, the $v_S$ satisfies
\[
d_{\scriptscriptstyle(S,\lambda)}\chi(v_S)>0
\]
for every real eigenvalue $\lambda$ of $S$. By continuity, we can find a neighbourhood $U_S$ of $S$ and an extension $V_S$ of $v_S$ such that for every $S'\in U_S$ and $\lambda'$ real eigenvalue of $S'$, we have $d_{\scriptscriptstyle(S',\lambda')}\chi(V_S)>0$.

The open sets $\{U_S\}_{S\in \Sp_1(n-1,\omega)}$ form an open cover of $\Sp_1(n-1,\omega)$. Choosing a countable subcover $\{U_{S_i}\}_i$ with a subordinate partition of unity $\{\lambda_i\}_i$, the vector field
\[
V=\sum_i\lambda_i V_{S_i}
\]
has the required properties.
\end{proof}

Proposition \ref{P:map-surj} implies that $\mathcal{G}_1$ is a smooth hypersurface in $\mathcal{G}$. Consider now the projection $\pi:\mathcal{G}_0\to \RR$, sending $(Q,\lambda)$ to $\lambda$, and let $\mathcal{G}_0=\pi^{-1}(1)$.
\begin{lem}\label{L:G0inG1}
The map $\pi:\mathcal{G}_1\to \RR$ is a submersion around $\mathcal{G}_0$.
\end{lem}
\begin{proof}
It is enough to find, for every point $(Q,1)\in \mathcal{G}_1$, a vector $v\in T_{\scriptscriptstyle(Q,1)}\mathcal{G}_1$ such that $d_{\scriptscriptstyle(Q,1)}\pi(v)\neq 0$. Fixing $(Q,1)$, we know in particular that $1$ is an eigenvalue of $Q$ and therefore $Q$ can be written, in some basis, as $Q=\diag(P,R)$ where $P\in \Sp(a,\omega)$, $R\in \Sp(n-1-a,\omega)$ and moreover
\[
P=\left(\begin{array}{cc}U^{-1} & 0 \\U^{\textrm{tr}}T & U^{\textrm{tr}}\end{array}\right)
\]
with $T$ symmetric and $U$ unipotent such that
\[
U^{\textrm{tr}}=\left(\begin{array}{ccc}1 & \cdots & 1 \\ & \ddots & \vdots \\ &  & 1\end{array}\right).
\]
For some $t$ small, let
\[
P(t)=\left(\begin{array}{cc}e^{-t}U^{-1} & 0 \\e^tU^{\textrm{tr}}T & e^tU^{\textrm{tr}}\end{array}\right), \quad Q(t)=\diag(P(t),R).
\]
Then the path $(Q(t),e^{-t})$ lies in $\mathcal{G}_0$ for small $t$, and $\pi(Q(t),e^{-t})=e^{-t}$. In particular, letting
\[
v={d\over dt}|_{t=0}(Q(t),e^{-t})
\]
we obtain $d_{\scriptscriptstyle(Q,1)}\pi(v)\neq 0$ thus proving the Lemma.
\end{proof}

The following Corollary is straightforward
\begin{cor}
The subset $\pi^{-1}\big((0,1]\big)\In \mathcal{G}_1$ is a smooth manifold, with boundary $\mathcal{G}_0$.
\end{cor}

\bibliographystyle{amsplain}

\hspace*{1em}\\
\begin{footnotesize}
\hspace*{0.3em}{\sc University of M\"unster,
Einsteinstrasse 62, 48149 M\"unster, Germany}\\
\hspace*{0.3em}{\em E-mail addresses: }{\sf mrade\_02@uni-muenster.de, wilking@uni-muenster.de}
\end{footnotesize}
\end{document}